\documentclass[11pt,a4paper]{article}
\usepackage[utf8]{inputenc}
\usepackage[T1]{fontenc}
\usepackage[french, english]{babel}
\usepackage{amsfonts}
\usepackage{amsmath, amsthm, mathrsfs}
\usepackage{bbm}
\usepackage{graphicx}
\usepackage[most]{tcolorbox}
\usepackage{xcolor}
\usepackage{pifont}
\usepackage{array}
\usepackage[left=2.5cm,right=2.5cm,top=2.5cm,bottom=2.5cm]{geometry}
\usepackage{caption}
\usepackage{enumerate}
\usepackage{subcaption}
\usepackage{todonotes}
\usepackage{relsize}
\usepackage{stmaryrd}

\usepackage[colorlinks=true,linkcolor=purple,citecolor=magenta]{hyperref}

\newcommand{\E}{\mathbb{E}}
\newcommand{\N}{\mathbb{N}}
\newcommand{\R}{\mathbb{R}}
\newcommand{\Z}{\mathbb{Z}}
\newcommand{\T}{\mathbb{T}}
\newcommand{\un}{\mathbbm{1}}
\newcommand{\PP}{\mathbb{P}}
\newcommand{\bP}{\textbf{P}}
\newcommand{\Q}{\mathbb{Q}}
\renewcommand{\P}{\mathbb{P}}
\newcommand{\Prob}{\mathbb{P}}
\newcommand{\FEP}{\textsc{fep}}
\newcommand{\FZR}{\textsc{fzr}}
\newcommand{\SWT}{\textsc{swt}}
\newcommand{\SSEP}{\textsc{ssep}}

\newcommand{\abs}[1]{\left\vert#1 \right\vert}

\newcommand{\pa}[1]{\left(#1\right)}
\newcommand{\gene}{\mathscr{L}_N}
\newcommand{\TB}{\mathcal{T}_{\mathsmaller{\mathsmaller{\mathsmaller{\mathsmaller{-}}}},N}}
\newcommand{\TG}{\mathcal{T}_{\mathsmaller{\mathsmaller{\mathsmaller{\mathsmaller{+}}}},N}}

\newcommand\ent[1]{\left\lfloor #1 \right\rfloor} 

\newcommand{\noteblue}[1]{\textcolor{MSBlue}{#1}}

\newtheorem{theorem}{Theorem}[section]
\newtheorem{lemma}[theorem]{Lemma}
\newtheorem{proposition}[theorem]{Proposition}
\newtheorem{corollary}[theorem]{Corollary}

\newtheorem{remark}{Remark}
\definecolor{MSBlue}{rgb}{.15,.0.35,.85}

\newcommand{\eqs}[1]{\begin{equation}#1\end{equation}}
\newcommand{\eq}[2]{\begin{equation}\label{#1}#2\end{equation}}

\title{Cutoff for the transience and mixing time of a SSEP with traps and consequences on the FEP}
\author{Cl\'ement Erignoux, Brune Massoulié}
\date{}

\begin{document}

\maketitle


\begin{abstract}
We introduce a new particle system that we call the SSEP with traps, which is non reversible, attractive, and has a transient regime.
We study its \emph{transience time} $\theta_K$, meaning the time after which the system is no longer in a transient state 
with high probability, on the ring with $K$ sites.  We first show that $\theta_K$ is of order $K^2 \log K$ for a system of size $K$, and more precisely that it exhibits a cutoff at time $\frac{1}{\pi^2} K^2 \log K$. We then show that its \emph{mixing time} also undergoes cutoff at the same time.

We further define a new mapping between the SSEP with traps and the Facilitated Exclusion Process (FEP) which has attracted significant scrutiny in recent years \cite{BBCS18,blondel_hydrodynamic_2020,GLS21}. We expect that this mapping will be a very useful tool to study the FEP's microscopic and macroscopic behaviour. In particular, using this mapping, we show that the FEP's transience time also undergoes a cutoff at time $\frac{1}{4 \pi^2} N^2 \log N$. Notably, our results show that for a FEP with particle density strictly greater than $\frac12$, the transient component is exited in a diffusive time. This allows to extend the upper-bound from \cite{ayre_mixing_2024} for the mixing time of the FEP with particle density $\rho > 1/2$.

\end{abstract}

\section*{Introduction}

In recent years, multi-phased and multi-type interacting particle systems have attracted a lot of attention from the statistical physics community as toy models for phase separation.  Some of these models belong to the family of lattice gases, whose particles evolve and interact on a discrete lattice \cite{liggett_interacting_2004}. Under particular scrutiny in the last ten years, the Facilitated Exclusion Process (FEP, see e.g. \cite{BBCS18,blondel_hydrodynamic_2020,CZ18,GLS21}), as well as the Activated Random Walks (ARW, e.g. see \cite{cabezas_2014,LL23,Rolla20,RS11,ST17}) model, can be seen as active/inactive multi-phased lattice gases. Because they are conservative, such models do not belong to the classical \emph{Directed Percolation} universality class \cite{RPV00}.

\medskip

In one dimension, the FEP, first introduced by physicists in \cite{BM09}, is a standard exclusion process, except that particle jumps to an empty neighbouring site can only occur if the jumping particle's other neighbouring site is occupied. This constraint makes the FEP non-reversible and non-attractive. Under this dynamics, isolated particles cannot move until another particle comes around, whereas isolated empty sites remain isolated forever.  In particular, after a transience time, the FEP either freezes because all particles become isolated (which can happen at density $\rho\leq 1/2$) or becomes ergodic and diffusive if all empty sites become isolated ($\rho\geq 1/2$) \cite{blondel_hydrodynamic_2020}.  The Facilitated Exclusion Process has rich mapping features: in the periodic setting, its evolution can be paired with an attractive and degenerate Facilitated Zero-Range Process (FZR, see \cite{blondel_hydrodynamic_2020}), in which any number of particles can occupy a site, and at rate $1$, a particle jumps away if it is not alone on the site. Activated random walks (as well as the \emph{particle-hole model} introduced in \cite{cabezas_2014}) have roughly the same behaviour as the FZR, except that particles jump independently, and become inactive at a given rate, rather than instantly, when they are alone on a site. In these models, diffusion or particle displacement is \emph{facilitated} by the presence of other particles close by, so that both are roughly diffusive at high densities, and frozen at low densities. The strong local interaction of these activated/facilitated models makes the derivation of their macroscopic and microscopic features challenging. For activated random walks, the particle's discrete trajectories are independent random walks, opening up the way to multi-scale arguments \cite{RS11} to study the macroscopic spread of particles. For the FEP, this is not the case, but its local stationary states have explicit expressions charging the ergodic and frozen components, so that the classical theory of hydrodynamic limit can be adapted to obtain extensive results on its phase-separated macroscopic behaviour, both in symmetric and asymmetric cases \cite{blondel_hydrodynamic_2020,blondel_stefan_2021,ESZ23-3,EZ23}. However, obtaining sharp quantitative estimates on the transience and mixing properties of such models is by no means straightforward.

\medskip

The FEP's transience time (after which it either becomes frozen or ergodic) on the ring $\T_N$ of size $N\geq 1$ was first studied in  \cite{blondel_hydrodynamic_2020, blondel_stefan_2021}, where it is shown that starting from a product state with density $\rho\neq 1/2$, the transience time is w.h.p. at most of order $(\log N)^q$ for some explicit constant $q>0$. The case of a deterministic initial configuration was then considered in \cite[Section 4.3]{ayre_mixing_2024}, where it is shown that if in the initial configuration, a fixed number $m$ of ergodic segments contain more than $N-K$ particles, the transience time is $\mathcal{O}(N^2\log N)$. As exploited in \cite{GLS21, ayre_mixing_2024}, after its transience time, the supercritical FEP (meaning a FEP with more than $N/2$ particles) becomes ergodic, at which point it can be mapped to the standard SSEP. For this reason, estimating the FEP's transience time is a crucial step to understand its mixing time, since the SSEP's mixing time has been thoroughly studied \cite{lacoin_simple_2017}. 

\medskip

In this article, we introduce a new phase separated model, that we call SSEP with traps (SWT), in which particles evolve on $\T_K$ according to a SSEP until they fall into a trap where they become inactive. Each trap can only contain a fixed number of particles, so that traps can ultimately disappear and then behave as regular SSEP empty sites. Our SWT model is attractive, and until either all particles or all traps have disappeared, it remains in a transient state where both particles and traps coexist. We are interested in the transient properties of the SWT, and show that its transience time undergoes a cutoff at time $t_K^\star:=K^2\log K/\pi^2$. This means that the probability $p_\xi(t)$ for the SWT to remain transient up to time $t$, starting from the worst possible transient configuration $\xi$, goes from $0$ to $1$ in a sharp window of size $\delta_K\ll t_K^\star$ around $t_K^\star$.  Our proofs rely on delicate coupling arguments between the \emph{periodic} SWT and the \emph{non-periodic} SSEP in contact with empty reservoirs. Note that, unlike the latter, the SWT does not preserve negative dependence (see Appendix \ref{subsec:noND}), so that arguments in the spirit of \cite{salez_universality_2022} cannot be straightforwardly adapted. However, we are able to make use of attractiveness to compare SWT with different numbers of particles and traps. 

\medskip

Once in the ergodic phase, the SSEP with traps behaves exactly like a SSEP. Because of that, we can study the SWT's mixing time by combining our  transience time estimates with SSEP mixing time estimates from \cite{lacoin_simple_2017}. 
Strikingly, we show that transience time and mixing time \emph{after the transience time} balance each other out (cf. Remark \ref{rem:mixingtransience}), in the sense that a transient supercritical SWT configuration with high transience time $t\simeq t_{K}^\star$ (e.g. a critical configuration with $K-1$ particles) will have a low mixing time $\tau=\mathcal{O}(K^2)$ once it reaches the ergodic component, and a supercritical configuration with low transience time $t=\mathcal{O}(K^2)$ (e.g. a configuration with $\delta K$ excess particles) will have a long mixing time $\tau=\mathcal{O}(K^2\log K)$ afterwards. In other words, being far from criticality shortens the transience time, but lengthens the mixing time. From this detailed study we can control the mixing time starting from any SWT configuration of size $K$ and find that the mixing time also exhibits cutoff around time $t_{K}^{\star}$.

\medskip
Then, developing a mapping between the SWT and the FEP's trajectories, we show that the transience time for the Facilitated Exclusion Process on the ring $\T_N$, also exhibits cutoff at time $t_{N/2}^\star$. This is a significant improvement on the previous estimates of the FEP's transience time \cite{ayre_mixing_2024, blondel_hydrodynamic_2020, blondel_stefan_2021}, since we lift all assumptions on the initial state, and obtain a sharp estimate, through the explicit constant $1/4\pi^2$ appearing in $t_{N/2}^\star$. 
Like the mapping to the zero-range process, this new mapping to the SWT is very useful: it also maps the FEP to an attractive process, however with the added upside that the SWT behaves in many ways similarly to the SSEP. In particular, unlike in the FZR, individual particles in the SWT behave as rate $1$ symmetric random walkers, for which very precise estimates are available.

\smallskip

Our mapping generalises to the circle the lattice path construction developed in \cite{ayre_mixing_2024} in the following sense. In \cite{ayre_mixing_2024}, a FEP on a segment is transformed into a height function that has corner-flip like dynamics, and by monotonicity it can be compared to the maximal and the minimal possible height functions: once they have joined, the height function has coupled with the stationary law. In \cite{ayre_mixing_2024}, this construction is only used on the segment, indeed this kind of argument would fail on the circle, since height functions can move up or down, and the maximal and minimal height functions don't necessarily meet. On the other hand, our mapping from the FEP to the SWT is robust enough to deal with the case of the circle. It also gives a new interpretation of the lattice path from \cite{ayre_mixing_2024} as the height function associated with the SWT.

The transience of the FEP is equivalent to the transience of the corresponding SWT. However, the mapping from the FEP to the SWT doesn't allow us to directly obtain the FEP's mixing time from the SWT's, as opposed to the transience time. Indeed, as laid out in Section \ref{sec:FEPmixing}, some information is lost as one goes through the mapping, so we can easily obtain a lower bound on the FEP's mixing time but not an upper bound.
Although we don't have an upper bound for general starting configurations, we can generalise \cite[Theorem 2.3.]{ayre_mixing_2024} and show that for all sequence of number of particles $K(N)$ such that $K(N)/N \longrightarrow \rho > \frac12$, the mixing time for all FEP configurations with $K(N)$ particles is less than $C_{\rho,\varepsilon} N^2 \log N$. Controlling the mixing time of the FEP for any initial configuration would require a more detailed study, and we leave this for future work.

\medskip

 The SWT introduced here is of deep interest in its own right, and is easily defined in more general settings (higher dimension, general lattices, any particle jump kernel). It can be seen as a generalisation of a phase-separated model introduced by Funaki \cite{Funaki99}, and is reminiscent of another trap model introduced in \cite{CFRS23}, with the important difference that particles cannot leave traps, so that our model is non-reversible. To the best of our knowledge, the cutoff phenomenon has so far been identified overwhelmingly in the context of Markov process's mixing time. It was first discovered in \cite{diaconis_generating_1981, aldous_shuffling_1986} while studying card shuffling, and the cutoff of mixing time for many Markov chains has been established since, see \cite[Chapter 18]{levin_markov_2017} for an introduction to this phenomenon. We study here a novel, and natural, instance of cutoff, which affects kinetically constrained models with a transient regime. 

\medskip

This article is organised as follows. In Section \ref{sec:swt}, we introduce the SSEP with traps model that will be studied throughout. We also present our first main results, Theorems \ref{thm:transiencetimeSSEP} and \ref{thm:cutoff_SWT}, which state that the SWT's maximal transience time is of order $\mathcal{O}(K^2\log K)$ and undergoes cutoff at time $t_K^\star$, and Theorem \ref{thm:cutoff_mix}, which states that the SWT's mixing time also exhibits cutoff at time $t_K^{\star}$. We also present some important properties of the SWT. In Section \ref{sec:fep}, we define the Facilitated Exclusion Process, translate our transience time cutoff results to the latter (Theorem \ref{thm:freezetime_cutoff_crit}) and state our upper bound on the mixing time, Proposition \ref{prop:mix_fep}. We then present our mapping between SWT and FEP, which is the main argument to obtain Theorem \ref{thm:freezetime_cutoff_crit} from  Theorem \ref{thm:cutoff_SWT}. In Section \ref{sec:proof_tran_swt}, we introduce a  labelled SWT that will be used throughout, prove with it Theorem \ref{thm:transiencetimeSSEP}, show Theorem \ref{thm:cutoff_mix} and fully estimate the SWT's mixing time in Proposition \ref{prop:t_mix_swt_all_s}. We prove the cutoff for the SWT's transience time in Section \ref{sec:cutoffSWT}, based on coupling arguments with the non periodic SSEP in contact with empty reservoirs, and fully estimate the super-critical transience time in Lemma \ref{lem:t_tran_critique}. Finally, in Section \ref{sec:proofthm2.2}, we introduce the Facilitated Zero-Range process, and use both the latter and the SWT to extend our transience time cutoff results to the FEP. 

\subsection*{Acknowledgements} The authors would like to thank Anna Benhamou, Pietro Caputo, Paul Chleboun, Justin Salez, Cristina Toninelli and Hong-Quan Tran for interesting discussions on the cutoff phenomenon.

\subsection*{Notation and conventions}
Throughout this article, we will repeatedly map processes into each other. In order to give as clear an exposition as possible, we will whenever possible  always use the same notational conventions for the three main processes we will investigate. We present those notations here for reference.
\begin{center}
\def\arraystretch{1.4}
\begin{tabular}{l|c|c|c}
Process &SSEP with traps& Facilitated Exclusion& Facilitated Zero-range\\
\hline
\hline
Acronym & \textsc{SWT} & \textsc{FEP} & \textsc{FZR} \\
Scaling parameter &$K$ & $N$& $P$ \\
Discrete variable & $k$ & $x$ & $y$ \\
Discrete configuration & $\xi=(\xi_k)_{k\in \T_K}$ & $\eta=(\eta_x)_{x\in\T_N} $ & $\omega=(\omega_y)_{y \in \T_P}$  \\
Statespace &$ \Gamma_K$ &$\Sigma_N$& $\Omega_P $\\
Law of the process & ${\bf P}^K_\xi$  & $\P^N_\eta$ & ${\bf P}^P_\omega$ \\
\end{tabular}
\end{center}
Furthermore, to ease reading, whenever a new notation that is relevant throughout is introduced inside of a paragraph, it will be coloured \noteblue{in blue}. Other notation conventions will be used throughout;

\begin{itemize}
    \item We use $"\llbracket"$ to delimit sets of \emph{integers}, meaning for example that  $\llbracket a, b\rrbracket=\{a,a+1,\dots,b-1,b\}$, whereas   $\llbracket a, b\llbracket=\{a,a+1,\dots,b-1\}$.
\item We denote by $\N:=\{0,1,\dots\} $ the set of non-negative integers.
\item We will consider various time-dependent processes throughout this article. As a general rule, for $\{\gamma(t),\; t\geq 0\}$, we will simply denote by $\gamma:=\gamma(0)$ its \emph{initial state}, and by $\gamma(\cdot)$ the whole process. For example, ${\bf P}^K_\xi(\xi(t)=\xi')$ is the probability, starting from $\xi$, to be in a  state $\xi'$ at time $t$. 
\end{itemize}

\section{The SSEP with traps: model and results}
\label{sec:swt}
\subsection{Definition}
\label{sec:defxi}
We introduce our model, that we call \emph{SSEP with traps}, and for which we use the abbreviation \emph{SWT}. We fix a positive integer $K\geq 1$, and define the SWT on the \emph{periodic} ring $\noteblue{\T_K:=\{1,\dots K\}}$, and define its statespace as
\eq{eq:defGammaK}{\Gamma_K:=\{n\in \Z,\; n\leq 1\}^{\T_K}.} 
As represented in Figure \ref{fig:SWTconfig}, a configuration $\noteblue{\xi=(\xi_k)_{k\in \T_K}}\in \Gamma_K$ is interpreted as follows:
\begin{itemize}
\item For $k\in \T_K$, $\xi_k=1$ indicates that site $k$ is occupied by a particle.
\item Similarly, $\xi_k=0$ indicates that site $k$ is empty.  An empty site can alternatively be interpreted as a trap of depth zero. 
\item Finally, $\xi_k<0$ indicates that at site $k$, there is a trap of depth $|\xi_k|$.
\end{itemize}

The SWT $\xi(\cdot):=\{\xi(t),\; t\geq 0\}$ is a continuous time Markov process on $\Gamma_K$, whose dynamics is analogous to that of the classical SSEP. Particles jump symmetrically at rate 1 to a neighbouring site with no particle (exclusion rule), with the following exception; as represented in Figure \ref{fig:SWTconfig}, when a particle jumps towards a trap of positive depth, it falls into it, becomes inactive forever, and the trap's depth reduces by one. Since they play no further roles in the dynamics, once the trap's depth's decrease has been recorded, inactive particles are considered \emph{dead}, so that when referring to a \emph{particle} in the SWT, unless otherwise specified we will always be talking about a \emph{live} particle, meaning the single live particle present at a site $k$ where $\xi_k=1$. According to this dynamics, traps get progressively filled up by successive particles falling into them, until they are full, at which point they start behaving as empty sites. Since dead particles are inactive, trap depth can only decrease, in particular the SWT is \emph{non-reversible}. The SWT is therefore driven by the generator 
\eq{eq:genexi}
{
{\mathscr{L}}^{\SWT}_K f (\xi)=\sum_{k\in \T_K}\sum_{z=\pm 1}{\bf 1}_{\{\xi_k=1, \xi_{k+z}\leq 0\}}\{f(\xi^{k,k+z})-f(\xi)\},
}
where 
\eqs{
\xi^{k,k+z}_{k'}=\begin{cases}0 & \mbox{ if }k'={k}\\
\xi_{k+z}+1 & \mbox{ if }k'=k+z\\
\xi_{k'}& \mbox{ if }k'\neq k,k+z.
\end{cases}
}

Given a configuration $\xi\in \Gamma_K$, we define its \emph{total trap depth}, (resp. its \emph{number of particles}), as the quantity \eq{eq:defTTD}{|\xi^-|:=-\sum_{k\in\T_K}\min(\xi_k, 0) \qquad \Big(\mbox{resp.} \qquad  |\xi^+|:=\sum_{k\in\T_K}\un_{\{\xi_k=1\}}\Big).}
We also define its \emph{number of excess particles} $\noteblue{S(\xi)=\sum_{k\in \T_K}\xi_k=|\xi^+|-|\xi^-|}\in \Z$.
Note that the number of excess particles $S(\xi)$ is conserved throughout the dynamics, because the process loses a particle iff it loses one trap depth. It can also be negative if the configuration has more total trap depth than particles.

\begin{figure}
\centering
\includegraphics[width=10cm]{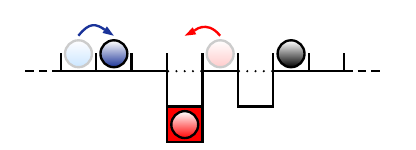}
\caption{A local configuration of the SWT, with two jumps represented: the jump of the blue particle to an empty neighbour, and the jump of the red particle to a neighbouring trap, killing it and reducing the trap's depth by $1$.}
\label{fig:SWTconfig}
\end{figure}

\medskip

The SWT  can be seen as a generalisation of a model previously introduced by Funaki in \cite{Funaki99}, in which sites can either be occupied by water ($\xi_k=1$) or ice ($\xi_k=-\ell$), and water is destroyed by ice but melts it after $\ell$ visits by water.  That process, whose macroscopic limit derived in \cite{Funaki99} is a free boundary problem,  is actually a SSEP with traps, but started from a distribution where only occupied sites and traps of depth exactly $\ell$ are allowed.

\medskip 

For our purpose, we introduce the SWT as a symmetric process on the 1-dimensional ring $\T_K$. However, this model  readily extends to more general setups, namely to different lattices, and to general local jump rates (e.g. asymmetric or weakly asymmetric).

\subsection{Transient, frozen and ergodic states for the SWT}
\label{subsec:tran_fro_erg_swt}
Because of the presence of traps, depending on the sign of its number of excess particles $S(\xi(0))$,  the SWT will eventually either freeze, because all of its particles are trapped, or become a SSEP, because all its traps have been filled up.

\medskip

We call \emph{critical} (resp. \emph{subcritical}, \emph{supercritical}) configuration a configuration $\xi$ with $S(\xi)=0$ (resp. $S(\xi)\leq 0$, $S(\xi)\geq 0$). Note that we identify \emph{critical} configurations as both \emph{subcritical} and \emph{supercritical}. According to the previous remark, we distinguish two categories of \emph{subcritical configurations};
\begin{itemize}
\item A configuration with no particles is frozen, and we denote by 
\eq{eq:FK}{\mathcal{F}^{\SWT}_K:=\{\xi\in\Gamma_K \mbox{ such that }  \; \xi_k \leq 0, \;\forall k\in \T_K \}}
the set of frozen configurations.
\item Any  subcritical configuration that is not frozen will be called \emph{subcritical transient}, and we denote by 
\eqs{\mathcal{T}_{-,K}^{\SWT}:=\big\{\xi\in \Gamma_K\setminus\mathcal{F}^{\SWT}_K\mbox{ such that } S(\xi)\leq 0\big\}}
the set of subcritical transient configurations. 
\end{itemize}

\bigskip

We have an analogous classification for \emph{supercritical configurations}.
\begin{itemize}
\item Once all traps have been filled up, the process becomes a classical SSEP, and in particular becomes ergodic and reversible. We therefore define the \emph{ergodic component} for the process, which is the set 
\eq{eq:EK}{\mathcal{E}_K^{\SWT}:=\{\xi\in\Gamma_K \mbox{ such that }  \; \xi_k \geq 0,\;\forall k\in \T_K \}=\{0,1\}^{\T_K}}
of classical SSEP configurations.
\item Finally, any supercritical configuration that is not ergodic will be called \emph{supercritical transient}, and we denote by 
\eq{eq:DefTplusK}{\mathcal{T}_{+,K}^{\SWT}:=\big\{\xi\in \Gamma_K\setminus\mathcal{E}_K^{\SWT}  \mbox{ such that }  \; S(\xi)\geq  0\big\}}
the set of supercritical transient configurations.

\end{itemize} 
Note that the only configuration that is both \emph{ergodic} and \emph{frozen} is the empty configuration $\xi\equiv 0$. We call \emph{transient state} any configuration that is neither ergodic nor frozen, and denote by 
\eqs{\mathcal{T}_K^{\SWT}=\{\xi\in \Gamma_K \mbox{ such that }  \; |\xi^-| > 0,\; |\xi^+| > 0\} } 
the set of transient configurations.

\subsection{Transience and mixing time for the SWT}

Given an initial configuration $\xi\in \Gamma_K$, we denote by \noteblue{${\mathbf P }_\xi^K$} the distribution of the SSEP with traps $\xi(\cdot)$ on $\T_K$, driven by the generator $\mathscr{L}_K^{\SWT}$ defined in \eqref{eq:genexi} and started from the initial configuration $\xi$. 

\subsubsection{Transience time}

We introduce the SWT's transience probability at time $t$, starting from $\xi$,
\eq{eq:pxit}{p_\xi(t):={\bf P}^K_\xi\big(\xi(t)\in \mathcal{T}_K^{\SWT}\big).}
Note that since it is encoded implicitly in the initial configuration $\xi$, we do not make the scaling parameter $K$ appear in $p_\xi(t)$.
We now define the maximal transience  probability, and the $\varepsilon$-transience time, 
\eq{eq:pkSWT}{p_{K}^{\SWT}(t):=\sup \Big\{p_\xi(t): \xi \in \Gamma_K\Big\}\qquad \mbox{ and }\qquad \theta_{K}^{\SWT}(\varepsilon):=\inf \Big\{t\geq 0: p_{K}^{\SWT}(t)\leq \varepsilon\Big\}.}
It is not hard to check that $p_{K}^{\SWT}=[0,+\infty)\to[0,1)$ is differentiable and strictly decreasing, so that $\theta^{\SWT}_{K}$ is its inverse function. However, we rather define $\theta_K^{\SWT}(\varepsilon)$ as in \eqref{eq:pkSWT}, since it is the standard presentation for cutoff phenomena and allows us to avoid proving the last statement.
Our first result is an estimate of the transience probability for the SWT, uniformly in $K$.

\begin{theorem} 
\label{thm:transiencetimeSSEP}
We have 
\eq{eq:transiencetimeSSEP}{\lim\limits_{t\to+\infty}\sup_{K\ge 0}p_{K}^{\SWT}(tK^2 \log K)=0.}
\end{theorem}

 We now state our second main result, namely that the transience time for the SWT  undergoes a cutoff around time 
\eq{eq:tNstar}{t_K^\star:=\frac{K^2\log K}{\pi^2}.}
\begin{theorem}[Cutoff for the transience time of the SSEP with traps] 
\label{thm:cutoff_SWT}
There exists $C>0$ such that for all $K \in \N^*$ and $0<\varepsilon<1$,
\begin{equation}
\label{eq:encadrement_t_freeze}
\left|\theta_{K}^{\SWT}(\varepsilon) - t_K^\star \right| \le C K^2 \left( 1+ \log \frac{\log K}{\varepsilon \wedge (1-\varepsilon)}\right).
\end{equation} 
In particular,  the transience time exhibits cutoff, meaning that for all $0<\varepsilon < 1$,
\begin{equation}\label{eq:cutoff}
\lim\limits_{K\to \infty} \frac{\theta_K^{\SWT}(\varepsilon)}{\theta_K^{\SWT}(\frac14)} = 1.
\end{equation}
\end{theorem}

Theorem \ref{thm:transiencetimeSSEP} will be proved later in Section \ref{sec:proof_th_3}. The proof of Theorem \ref{thm:cutoff_SWT} is the purpose of Section \ref{sec:cutoffSWT}. 

\medskip

Note that Theorem \ref{thm:transiencetimeSSEP} could be obtained straightforwardly using Theorem \ref{thm:cutoff_SWT}. However, we choose this presentation of both results because the proof of Theorem \ref{thm:transiencetimeSSEP} will lay the foundation for the (harder) proof of Theorem \ref{thm:cutoff_SWT}. Roughly speaking, the first theorem states that on time scales $tK^2\log K$, the transience probability goes from $1$ to $0$ over long times. Theorem \ref{thm:cutoff_SWT} states that this evolution actually occurs brutally, and at an explicit time $t:=1/\pi^2.$

\medskip

\begin{remark}[Worst configurations and optimal cutoff window]
\label{rem:worstcutoff}
We point out below in Remark \ref{rem:attract} that the maximal transience time for the SWT is actually achieved for critical configurations, meaning for configurations with as many particles as total trap depth. Among critical configuration however, it is not clear which one is the worst case scenario. 

\medskip

Simulations suggest that the configuration with a single deep trap is the worst one, and proving this claim would allow us to improve the upper bound and obtain a cutoff window of $\mathcal{O}(K^2)$ instead of $\mathcal{O}(K^2 \log \log K)$. However, no straightforward argument so far proves that the unique case trap is the worst one, so this is all left as a conjecture at this point.
\end{remark}

\begin{remark}[Transience time for independent random walks] \label{rem:IRW1}
If the particles' trajectories until they are trapped are independent random walks, it is possible to show, following the same steps, that Theorem \ref{thm:cutoff_SWT} is also true: the transience time exhibits cutoff at the same time $t_K^\star$. This comes from the fact that the time for a SSEP with reservoirs with $K-1$ particles to lose $s$ particles is of the same order as the time for $s$ independent random walks out of $K-1$ to reach empty reservoirs, so \eqref{eq:encadrement_theta} is also true for this process. 
\end{remark}

\subsubsection{Mixing time}

Our third main result concerns the mixing time of the SWT. We denote by $\noteblue{d_{\textsc{tv}}}$ the total variation distance between two probability measures. Given two integers $0\leq s\leq K$, denote by $\noteblue{\Gamma_{K,s}:=\{\xi\in \Gamma_K, \;S(\xi)=s\big\}}$ (resp. by $\noteblue{\Sigma_{K,s}:=\big\{\xi\in \{0,1\}^{\T_K}, \;|\xi|=s\big\}}$)  the set of SWT configurations (resp SSEP configurations) with $s$ excess particles. We then define 
\eqs{\pi_{K,s}(\xi)=\binom{K}{s}^{-1}{\bf 1}_{\{\xi\in\Sigma_{K,s}\}}}
the uniform distribution on $\Sigma_{K,s}$, which is invariant law of the SWT on $\T_K$ with $s$ excess particles.
For $\varepsilon > 0$, the $\varepsilon$-mixing time of the SWT is defined as 
\eq{eq:def_tmix}{\tau_{K}^{\SWT}(\varepsilon) = \inf \left\{t \ge 0 : \sup_{\substack{s\in \llbracket0,K\rrbracket\\\xi\in \Gamma_{K,s}}} d_{\textsc{tv}}\left(\bP_{\xi}^K\left(\xi(t) \in \cdot\right),\pi_{K, s}\right) \le \varepsilon \right\}.}

Then we have the following result:
\begin{theorem}[Cutoff for the mixing time of the SWT]
\label{thm:cutoff_mix}
There exists $C>0$ such that for all $K \in \N^*$ and $0<\varepsilon<1$,
\eq{eq:tmix_fenetre}{|\tau_K^{\SWT}(\varepsilon) - t_K^{\star}| \le CK^2\left(1 + \log \frac{\log K}{\varepsilon \wedge (1-\varepsilon)}\right). }
In particular, the mixing time for the SWT also exhibits cutoff at time $t_K^{\star}$, meaning that for all $0<\varepsilon < 1$,
\begin{equation}\label{eq:cutoff2}
\lim\limits_{K\to \infty} \frac{\tau_K^{\SWT}(\varepsilon)}{t_K^{\star}} = 1.
\end{equation}
\end{theorem}
\begin{remark}
\label{rem:mixingtransience}
At a first glance, it might appear surprising that the cutoff for the SWT's mixing time occurs at the same time $t_K^\star$ as the one for its transience time, even though to enforce mixing, one must surely first exceed the transience time. Actually, there is no contradiction here, because the upper bound $t_K^\star$ for the transience time is achieved (cf. Remark \ref{rem:worstcutoff}) for critical or close-to-critical configurations for which $S(\xi)=\mathcal{O}(1)$. For those, once the ergodic component is reached, mixing requires a time $\mathcal{O}(K^2)$, and therefore has no influence on the SWT's overall mixing time, whose dominant contribution is simply the transience time of order $\mathcal{O}(K^2\log K)$.

Conversely, SSEP configurations with  mixing time of order $\mathcal{O}(K^2\log K)$ need to have $\mathcal{O}(K)$ particles, and must therefore stem from a SWT configuration with $S(\xi)=\mathcal{O}(K)$ excess particles. But according to Lemma \ref{lem:t_tran_critique} below, any such SWT configuration reaches the ergodic component in a diffusive time of order $\mathcal{O}(K^2)$. In this case, the mixing time \emph{starting from the ergodic component} is the dominant contribution to the SWT's overall mixing time. Detailed estimates on the behaviour according to $S(\xi)$ are given in Proposition \ref{prop:t_mix_swt_all_s}.
\end{remark}

\begin{remark}[Mixing time for independent random walks]
As we did for the transience time, we compare the mixing time of the SWT to the mixing time of the process where particles perform independent random walks until they get trapped. We can show, see Appendix \ref{app:tmix}, that the mixing time of $s$ independent random walks when there are no traps is less than $\frac{K^2}{2 \pi^2} \log s$, so in particular the mixing time for independent random walks with traps is less than $t_K^{\star}$. Since the mixing time of this process is lower bounded by its transience time $t_K^{\star}$ (see Remark \ref{rem:IRW1}), it is equal to $t_K^{\star}$, the mixing time of the SSEP with traps. This result is quite interesting, since for the regular SSEP, \cite{oliveira_mixing_2013, hermon_exclusion_2020} have explored the link between the mixing time of the SSEP and of independent random walks on general graphs, and it is conjectured by Oliveira in \cite{oliveira_mixing_2013} that they should be of the same order. In our modified version of the SSEP, this appears to be true.
\end{remark}

\begin{remark}[No negative dependence for the SWT]
As the  SSEP and the SSEP in contact with reservoirs both  preserve negative dependence \cite{salez_universality_2022, tran_cutoff_2022}, and since traps behave in some ways as empty reservoirs, one would expect that the SWT does preserve negative dependence. Surprisingly, this is not the case, as detailed in Appendix \ref{subsec:noND}, and prevents us from exploiting the tools used in \cite{salez_universality_2022} in the context of the SSEP in contact with reservoirs. 

\end{remark}

Theorem \ref{thm:cutoff_mix} is proved in Section \ref{sec:proof_cutoff_mix}. Its proof is actually easier than estimating the transience time, since we just combine estimates on the transience time of the SWT and on the mixing time of the SSEP.

\subsection{Attractiveness and monotonicity of the transience probability}
\label{subsec:att_swt}

The SWT is attractive, meaning that given two SWT configurations $\xi\leq \xi'$ there exists a coupling ${\bf P}_{\xi, \xi'}^K$  between two SWT $\xi(\cdot)$, $\xi'(\cdot)$ respectively started from $\xi,$ $\xi'$ such that for any $t>0$,
\eq{eq:attractiveness}{{\bf P}_{\xi, \xi'}\big(\xi(t)\leq \xi'(t)\big)=1.}

\medskip
This coupling is easily defined, and is very similar to the classical \emph{basic coupling} for zero-range processes \cite[Chapter 2, Section 5]{KL99}. It is enough to equip both processes with the same Poisson clocks associated with the system's edges, and, whenever a clock rings on an edge $e\in \T_K$, make any possible particle jump over $e$ in both $\xi$ and $\xi'$. More precisely, shorten
\eqs{\delta_{k,z}={\bf 1}_{\{\xi_k=1, \xi_{k+z}\leq 0\}}\qquad \delta'_{k,z}={\bf 1}_{\{\xi'_k=1, \xi'_{k+z}\leq 0\}},}
which is either $0$ or $1$ depending on whether a particle can jump from $k$ to $k+z$ in $\xi$, $\xi'$. The basic coupling ${\bf P}_{\xi, \xi'}^K$ is then defined as the distribution of a Markov process on $\Gamma_K^2$, started from $(\xi, \xi')$ and  driven by the generator
\begin{multline*}
{\mathscr{L}}^{\SWT}_{K,2} f (\xi, \xi')=\sum_{k\in \T_K}\sum_{z=\pm 1}\bigg(\delta_{k,z}\delta'_{k,z}\Big\{f(\xi^{k,k+z},(\xi')^{k,k+z})-f(\xi, \xi')\Big\}\\
+\delta_{k,z}(1-\delta'_{k,z})\Big\{f(\xi^{k,k+z},\xi')-f(\xi, \xi')\Big\}+(1-\delta_{k,z})\delta'_{k,z}\Big\{f(\xi,(\xi')^{k,k+z})-f(\xi, \xi')\Big\}\bigg),
\end{multline*}
To see that order between configurations is preserved under this coupling, fix $k$, $k+z$, and assume that locally $\xi_k\leq \xi'_k$, $\xi_{k+z}\leq \xi'_{k+z}$ at time $t^-$, and that at time $t$ a clock triggers a jump from $k$ to $k+z$. Because the two configurations are locally ordered, only three cases are possible
\begin{itemize}
\item If $\xi'_k(t^-)=\xi'_{k+z}(t^-)=1$, nothing will happen at time $t$ in $\xi'$, so that $\xi'_k(t)=\xi'_{k+z}(t)=1$ and order is clearly preserved since for any $k'$, $\xi_{k'}\leq 1$ by definition.
\item If  $\xi'_k(t^-)=1$ and $\xi'_{k+z}(t^-)\leq 0$, then we must have  $\xi_{k+z}(t^-)\leq \xi'_{k+z}(t^-)\leq  0$ as well. Then, either  $\xi_k(t^-)=1$ and then a particle jumps in both processes at time $t$ towards site $k+z$, preserving the order, or $\xi_k(t^-)\leq 0$, and at time $t$ a particle only jumps in $\xi'$. If the latter occurs, after the jump, we have $\xi'_k(t)=0\geq \xi_k(t^-)=\xi_k(t)$, and $\xi'_{k+z}(t)=\xi'_{k+z}(t^-)+1 > \xi_k(t^-)=\xi_k(t)$, so that order is preserved as well.
\item If $\xi'_k(t^-)\leq 0$ and $\xi'_{k+z}(t^-)\leq 0$, nothing happens in either configuration, so that order is preserved as well.
\end{itemize}

\begin{remark}[Attractiveness and monotonicity of the transience probability]
\label{rem:attract}
Attractiveness is a very useful property, and particularly so regarding the transience time of the SWT : indeed, attractiveness yields that adding particles or removing traps to a supercritical SWT decreases  a.s. its transience time. Similarly, in a subcritical configuration, removing particles or adding traps decreases a.s. its transience time. In particular, recalling the notations introduced in Section \ref{subsec:tran_fro_erg_swt}, and  defining the subcritical, supercritical and critical transience probabilities, 
\eq{eq:critprobSWT}{p_{\pm,K}^{\SWT}(t):=\sup \Big\{p_\xi(t): \xi \in \mathcal{T}_{\pm,K}^K\Big\} \quad \mbox{and} \quad  p_{\star,K}^{\SWT}(t):=\sup \Big\{p_\xi(t), \xi \in \Gamma_K : S(\xi)= 0\Big\}.}
Then, the SWT's attractiveness yields (See Lemma \ref{lem:t_tran_K} below)  
\eq{eq:probcritSWT}{p_{+,K}^{\SWT}(t)= p_{-,K}^{\SWT}(t)= p_{\star,K}^{\SWT}(t).}
\end{remark}

\section{The facilitated exclusion process: model and results}
\label{sec:fep}

The SWT described above can be mapped into the Facilitated Exclusion Process (FEP) whose macroscopic behaviour was studied in \cite{blondel_hydrodynamic_2020, blondel_stefan_2021}.  
In particular, although Theorems \ref{thm:transiencetimeSSEP} and \ref{thm:cutoff_SWT} are interesting in their own right, they also yield analogous estimates on the (symmetric) FEP, which we now define.

\subsection{The facilitated exclusion process}

We define the FEP on the periodic lattice $\T_N = \{1,\dots,N\}$, it is an exclusion process with  state space is given by $\noteblue{\Sigma_N:=\{0,1\}^{\T_N}}$. We denote by \noteblue{$\eta \in \Sigma_N$} its configurations, i.e.  collections $(\eta_x)_{x\in \T_N}$, where $\eta_x=1$ means that site $x$ is occupied by a particle, and $\eta_x=0$ means that site $x$ is empty.  Dynamically speaking, in the FEP each particle jumps at rate $1$ to a neighbouring site, with two constraints:
\begin{itemize}
\item an \emph{exclusion constraint}, meaning that particles cannot jump on sites that are already occupied,
\item a \emph{kinetic constraint}, meaning that a particle at site $x$ can jump to a neighbouring empty site $x\pm 1$ if and only if its other neighbouring site $x\mp 1$ is occupied. In the rest of the article, particles satisfying the kinetic constraint, i.e. with a neighbouring particle, will be referred to as \emph{active particles}, as opposed to \emph{isolated particles,} which are surrounded by empty sites and for this reason cannot jump. Note that we choose a different terminology (\emph{active/isolated} particles) than for the SWT (\emph{live/dead} particles), because no one-to-one mapping will be available between SWT particles and FEP particles.
\end{itemize}
Given these two constraints, the FEP on $\T_N$ is characterised by its Markov generator $\gene, $ acting on function $f:\Sigma_N\to \R$ as
\eq{eq:gene}{\gene^{\FEP} f (\eta)=\sum_{x\in \T_N}c_{x,x+1}(\eta)\{f(\eta^{x,x+1})-f(\eta)\},}
where the jump rate $c_{x,x+1}$ is given by 
\eqs{
c_{x,x+1}(\eta)=\eta_{x-1}\eta_x(1-\eta_{x+1})+(1-\eta_x)\eta_{x+1}\eta_{x+2},
}
and $\eta^{x,x+1}$ is the configuration where sites $x$ and $x+1$ have been inverted in $\eta$, namely 
\begin{equation} \label{eq:eta_ech}
\eta^{x,x+1}_y=\begin{cases}\eta_x & \mbox{ if }y={x+1}\\
\eta_{x+1} & \mbox{ if }y=x\\
\eta_y& \mbox{ if }y\neq x,x+1.
\end{cases}
\end{equation}

\subsection{Transient, ergodic, frozen configurations}

\begin{figure}
    \centering
    \begin{subfigure}[b]{0.45\textwidth}
        \centering \includegraphics[width=\textwidth]{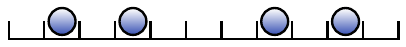}
        \caption{A frozen configuration}\label{fig:fep_gele}
    \end{subfigure}
        \hspace{0.08\textwidth} 
    \begin{subfigure}[b]{0.45\textwidth} 
        \centering \includegraphics[width=\textwidth]{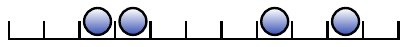}
        \caption{A subcritical transient configuration}\label{fig:tran_bad}
    \end{subfigure}
        		\begin{subfigure}[b]{0.45\textwidth} 
        		\vspace{0.5cm}
        \centering \includegraphics[width=\textwidth]{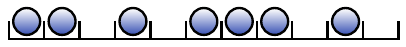}
        \caption{An ergodic configuration}\label{fig:fep_ergo}
    \end{subfigure}
                \hspace{0.08\textwidth} 
		\begin{subfigure}[b]{0.45\textwidth} 
		                        		\vspace{0.5cm}	
        \centering \includegraphics[width=\textwidth]{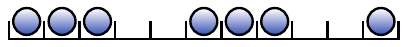}
        \caption{A supercritical transient configuration}\label{fig:tran_good}
    \end{subfigure}		
    \caption{Examples of the different types of configuration of the FEP on $\T_{11}$}\label{fig:ex_fep}
\end{figure}
The FEP is conservative, meaning its number of particles is preserved by the dynamics.  Given a configuration $\eta\in \Sigma_N$, denote by $\noteblue{K:=|\eta|}$ its number of particles and $\noteblue{P:=N-|\eta|}$ its number of empty sites. As for the SWT we distinguish four possible types of configurations, summarised in Figure \ref{fig:ex_fep}, depending on their number of particles. The sets of configurations are named in analogy to Section \ref{subsec:tran_fro_erg_swt}.

\bigskip

First, we distinguish two categories of \emph{subcritical configurations}, i.e. such that $K\leq N/2$.
\begin{itemize}
\item Because of the kinetic constraint, a configuration with no two neighbouring occupied sites is frozen. We denote by 
\eq{eq:FN}{\mathcal{F}^{\FEP}_N:=\{\eta\in\Sigma_N :  \; \eta_x\eta_ {x+1}=0 ,\; \forall x\in \T_N\}}
the set of frozen configurations. The configuration in Figure \ref{fig:fep_gele} is frozen.
\item Any configuration with $K\leq N/2$ particles that is not frozen will be called \emph{subcritical transient}, and we denote by 
\eqs{\TB^{\FEP}:=\Big\{\eta\in \Sigma_N\setminus\mathcal{F}_N^{\FEP}  :  \; |\eta|\leq N/2\Big\}}
the set of subcritical transient configurations. The configuration in Figure \ref{fig:tran_bad} is subcritical transient. Starting from a subcritical transient configuration, the process ultimately reaches a frozen configuration.
\end{itemize}

\bigskip

We now turn to \emph{supercritical configurations}, i.e. such that $K\geq N/2$
\begin{itemize}
\item Once again, because of the kinetic constraint, the process can break up pairs of neighbouring empty sites, but not create them. Indeed, in order for an empty site to ``jump'' next to another would require a particle in-between jumping out, which is not possible for the FEP. For this reason, we define the \emph{ergodic component} for the process, which is the set 
\eq{eq:EN}{\mathcal{E}_N^{\FEP}:=\{\eta\in\Sigma_N :  \; (1-\eta_x)(1-\eta_{x+1})=0 ,\; \forall x\in \T_N\}.}
We call \emph{ergodic configurations}  its elements, which are the set of configurations with isolated empty sites. The configuration in Figure \ref{fig:fep_ergo} is ergodic. 
\item Finally, any supercritical configuration ($K\geq N/2$) particles that is not ergodic will be called \emph{supercritical transient}, and we denote by 
\eqs{\TG^{\FEP}:=\Big\{\eta\in \Sigma_N\setminus\mathcal{E}_N^{\FEP}  :  \; |\eta|\geq N/2\Big\}}
the set of supercritical transient configurations. The configuration in Figure \ref{fig:tran_good} is supercritical transient. As in the subcritical case, starting from a supercritical transient configuration, the process ultimately reaches the ergodic component. 
\end{itemize}
In the case where $N$ is even, note that the alternated configurations $\circ \!\bullet\! \circ\! \bullet \dots \circ \!\bullet$ and $\bullet \!\circ\! \bullet\! \circ \dots\bullet\!\circ$ ($\circ$ representing empty sites and $\bullet$ representing occupied sites) are \emph{critical}, in the sense that they are both \emph{ergodic} and \emph{frozen}. In general, we call \emph{transient state} any configuration that is neither ergodic nor frozen, and denote by 
\eqs{\mathcal{T}_N^{\FEP}=\TB^{\FEP}\cup \TG^{\FEP}} 
the set of transient configurations for the FEP.

\subsection{Transience time for the FEP}

Given an initial configuration $\eta\in \Sigma_N$, we denote by $\noteblue{\Prob_\eta^N}$ the distribution of a FEP with generator $\gene^{\FEP}$  and started from the configuration  $\eta$. We denote by $\noteblue{\E_\eta^N}$ the corresponding expectation.

\medskip

As we were for the SWT, we are interested in the time the FEP takes to leave the transient component. Given $\eta \in  \Sigma_N$, as for the SWT we denote by
\eq{eq:petat}{p_\eta(t)=\Prob_\eta^N(\eta(t)\in \mathcal{T}_N^{\FEP})}
the probability that the process started from $\eta$ is still transient at time $t$. Note that this notation is the same as  in  \eqref{eq:pxit}, but this will not be an issue because of our use of different letters to represent the SWT ($\xi$) and the FEP ($\eta$). Consider the FEP's maximal transience probability and $\varepsilon$-transience time 
\eq{eq:def_time_eps}{p_{N}^{\FEP}(t):=\sup \Big\{p_\eta(t): \eta\in \Sigma_N\Big\} \qquad \mbox{and} \qquad \theta_{N}^{\FEP} (\varepsilon) = \inf \left\{t \ge 0 : p_N^{\FEP}(t) \le \varepsilon\right\},} 
analogous to the equivalent quantities for the SWT defined in \eqref{eq:pkSWT}. 
We start with a uniform estimate of the transience probability analogous to Theorem \ref{thm:transiencetimeSSEP} for the SWT, which will be proved at the end of the section.

\begin{theorem}
\label{thm:unif_transiencetime}
Uniformly in $N$ and in the initial configuration, if $t$ is large, the FEP is w.h.p. no longer transient at a time $t N^2 \log N$, meaning that 
\begin{equation}
\label{eq:time_unif}
\lim\limits_{t \to \infty} \sup_{N \ge 2}p_N^{\FEP}(tN^2\log N) = 0.
\end{equation}
\end{theorem}
Using the analogous estimate for the SWT, we prove this result at the end of the section.
Our next result is that the function $p_N^{\FEP}(\cdot)$ also goes through a cutoff at time $t_ {N/2}^\star$ (see \eqref{eq:tNstar}).
\begin{theorem}[Cutoff for the transience time of the  FEP] 
\label{thm:freezetime_cutoff_crit}
There exists $C>0$ such that for all $N \in \N^*$ and $0<\varepsilon<1$,
\begin{equation}
\label{eq:encadrement_t_freeze2}
\left|\theta_{N}^{\FEP}(\varepsilon) - t_{N/2}^\star \right| \le C N^2 \left(1 +\log \frac{\log N}{(1-\varepsilon) \wedge \varepsilon}\right).
\end{equation} 
In particular,  the transience time exhibits cutoff, meaning that for all $0<\varepsilon < 1$,
\begin{equation}\label{eq:cutoff_fep}
\lim\limits_{N\to \infty} \frac{\theta_N^{\FEP}(\varepsilon)}{\theta_N^{\FEP}(\frac14)} = 1.
\end{equation}
\end{theorem}
Theorem \ref{thm:freezetime_cutoff_crit} is proved below in Section \ref{subsec:super_crit_FEP}.

\medskip

\begin{remark}[Monotonicity and critical configurations] 
Given those uniform estimates of the maximal transience probability at time $t$, it is natural, as for the SWT in Remark \ref{rem:attract}, to try and identify the configuration(s) with yields the largest transience probability, meaning the $\eta^\star$ such that $p_{\eta^\star}(t)=p_N^{\FEP}(t).$ However, the answer to this question is by no means straightforward, mainly because of the lack of montonicity in the FEP's dynamics.

\medskip

To illustrate this point, consider the effect of adding or removing particles on  the transience probability. Removing particles in a subcritical configuration, one expects, should likely help reaching a frozen state, in the same way that adding particles to a supercritical configuration should likely  help reaching the ergodic component. However, in the absence of attractiveness arguments, those statements cannot be supported, and it is not clear, although we conjecture it, that the initial configuration that maximises the transience time should be a critical configuration, meaning $|\eta^\star|=N/2$ (for $N$ even).

\medskip

Another tricky question is that of the influence of the size of the system over the transience probability. Once again, we expect that as size $N$ of the system increases, the maximal transience probability $p_N^{\FEP}$ at fixed time $t$ should grow as well. This statement is not straightforwardly proved either, and is also left as a conjecture at this point.

\medskip

What is true however, and at the core of our proof of Theorem  \ref{thm:freezetime_cutoff_crit}, is that the transience probability starting from a configuration $\eta$ can be compared to the one started from a critical configuration $\widetilde{\eta}$, but on a smaller ring of size $\widetilde{N}=2(K(\eta)\wedge P(\eta))$, where $K$ and $P$ are respectively $\eta$'s number of particles and empty sites.This yields in particular that  for any time $t$,
\eq{eq:critprob}{p_{\star,N}^{\FEP}(t)\leq p_{N}^{\FEP}(t)\leq \sup_{\substack{N' \leq N \\N'\mbox{\tiny{even}}}}p_{\star,N'}^{\FEP}(t).}
However, because of the lack of monotonicity, it is not clear that the left and right-hand side above are equal.
\end{remark}

\begin{remark} 
\label{rem:fep_diffusif}
Although the ``worst'' transience phase lasts longer than the typical diffusive timescale (of order $N^2$), we  can show (see Remark \ref{rem:sctransience}  below) that for macroscopically supercritical configurations (ie $K(\eta) = \delta N$ with $\delta > \frac{1}{2}$), the transience time is indeed diffusive  instead of $\mathcal{O}(N^2\log N)$. 
\end{remark}

The proof of Theorems \ref{thm:unif_transiencetime} and \ref{thm:freezetime_cutoff_crit} for the FEP strongly rely on their counterparts for the SWT, thanks to a new mapping that we now formally introduce, and whose full construction will be detailed in Section \ref{sec:mapping}. 
\begin{theorem}
\label{thm:mapping}
For any $N\geq K\geq 1$, a FEP taking values in 
\eqs{\Sigma_{N,K}:=\{\eta\in \Sigma_N: |\eta|=K\}.}
can be mapped to a SWT on $\T_K$, meaning that there exists a deterministic dynamical mapping 
\eqs{\Pi^\star:D([0,+\infty), \Sigma_{N,K}) \to D([0,+\infty), \Gamma_K)}
between trajectories $\eta(\cdot)$ and $\xi(\cdot)$, such that 
\eqs{\eta(\cdot)\sim \Prob_{\eta}^N \qquad \Longleftrightarrow \qquad \Pi^\star[\eta]\sim {\bf P}_{\Pi^\star[\eta](0)}^K.}
Furthermore, this mapping preserves ergodic, frozen and transient components, meaning in particular that 
\eq{eq:tran_fep_swt}{\eta(t)\in \mathcal{T}^{\FEP}_N \qquad \Longleftrightarrow  \qquad \Pi^\star[\eta](t)\in \mathcal{T}^{\SWT}_K.}
\end{theorem}
Not that in order to keep notations simple, we denote the mapped trajectory by $\Pi^\star[\eta]$ rather than $\Pi^\star[\eta(\cdot)]$, but, of course, the latter depends on the whole trajectory $\eta(\cdot)$ and not only on its initial configuration. We also denote by
\eqs{\Pi(\eta):=\Pi^\star[\eta](0)}
the statical mapping between initial configurations.

\begin{proof}[Proof of Theorem \ref{thm:unif_transiencetime}]
Clearly, \eqref{eq:time_unif} is a consequence of Theorem \ref{thm:transiencetimeSSEP}. Indeed, given a trajectory $\eta(\cdot)$ of the FEP, according to Theorem \ref{thm:mapping}, $\xi(\cdot):=\Pi^\star[\eta]$  is a SSEP with traps on $K=|\eta|\leq N$ sites. Then, using that $t K^2 \log K\leq t N^2 \log N$ and \eqref{eq:tran_fep_swt}, we have
\eqs{ \eta(t N^2 \log N)\in \mathcal{T}^{\FEP}_N\quad \Longleftrightarrow \quad \xi(t N^2 \log N)\in \mathcal{T}^{\SWT}_K  \quad \Longrightarrow \quad \xi(t K^2 \log K)\in \mathcal{T}^{\SWT}_K .}
In particular, for any initial configuration $\eta\in \Sigma_{N,K}$

\begin{align*}
    p_{\eta}(tN^2 \log N)&= {\mathbf P}_{\Pi(\eta)}^K\big(\xi(tN^2\log N)\in \mathcal{T}^{\SWT}_K\big) \\
    & \le {\mathbf P}_{\Pi(\eta)}^K\big(\xi(tK^2 \log K)\in \mathcal{T}_K^{\SWT}\big) \\
    & \le \sup_{K\geq 0}p_{K}^{\SWT}(t K^2 \log K).
\end{align*}
Taking the supremum over all configurations $\eta\in \Sigma_N$ proves Theorem \ref{thm:unif_transiencetime}.
\end{proof}

\subsection{Mapping of the FEP to a SWT}
\label{sec:mapping}
We now build a mapping between the SWT and the FEP, and prove Theorem \ref{thm:mapping}. We start by the statical mapping that will allow us to map initial configurations, and then define a dynamical mapping between trajectories. This mapping generalises the ``lattice path'' construction, developed in \cite[Section 3.1]{ayre_mixing_2024} for the FEP on a segment, to the FEP on the circle. Indeed, Equation \eqref{eq:xieta} bears some resemblance with Equation (3.1) in \cite{ayre_mixing_2024}, which transforms a FEP on a segment into a height function. The novelty of our mapping is that we define it on the circle, and interpret it as a particle system (a SWT), which allows to study the FEP on a circle, when the height function approach could only be used on the segment.

\subsubsection{Statical mapping}

Fix an exclusion configuration $\eta$, that we will map to a SWT configuration $\xi$. Recall that we distinguish between the space variable $x$ for the FEP, and $k$ for the SWT. Assume that $\eta\in\Sigma_N$ is a non empty configuration of the FEP on $\T_N$, meaning that its total number of particles $\noteblue{K=K(\eta):=\abs{\eta}}\geq 1$ is not zero. We tag arbitrarily the first particle at or to the right  of  the origin in $\eta$ and denote by $\noteblue{X_1\in \T_N}$ its position, meaning that $\eta_0=\dots=\eta_{X_1-1}=0$. Denote similarly $\noteblue{X_1<X_2<X_3<\cdots <X_K}$ the successive positions in $\eta$ of particles to the right of the origin. Note in particular that we can trivially rewrite 
\eq{eq:etaXi}{\eta_x=\sum_{k=1}^K {\bf1}_{\{X_k=x\}},}
since in the right-hand side at most one indicator function is not $0$.
\medskip

As illustrated in Figure \ref{fig:stat}, we build the configuration $\xi\in \Z^{\T_K} $  by letting for all $k\in \T_K$, 
\begin{equation}
\label{eq:xieta}
\xi_k = 2 + X_k - X_{k+1}.
\end{equation}
\begin{figure}
\centering
\includegraphics[width=11.5cm]{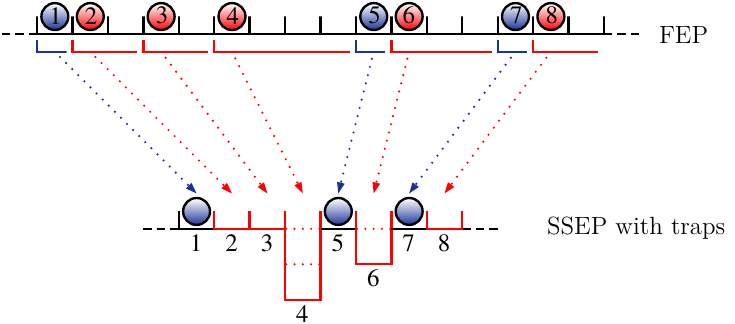}
\caption{Illustration of the statical mapping $\eta\mapsto \xi$, which associates particles in $\eta$ with sites in $\xi$. Particles that are followed by another particle in the $\eta$ are represented in blue and become occupied sites in $\xi$. Particles that are followed by one or more empty sites in $\eta$ are represented in red and become empty sites or traps in $\xi$.}
\label{fig:stat}
\end{figure}
    \begin{figure}
        \centering \includegraphics[width=10cm]{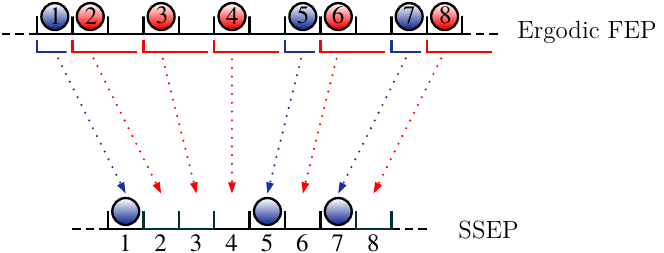}
        \caption{When a FEP configuration $\eta$ is ergodic, the mapped configuration $\xi=\Pi(\eta)$ is a SSEP configuration.}\label{fig:mapergo}
    \label{fig:ergomap}
    \end{figure}
\begin{figure}
        \centering \includegraphics[width=10cm]{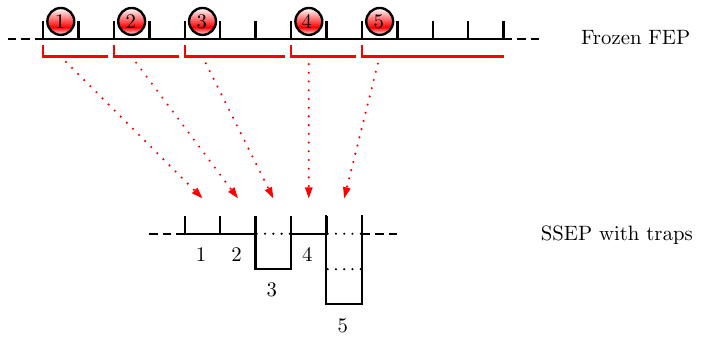}
        \caption{When a FEP configuration $\eta$ is frozen, the mapped configuration $\xi=\Pi(\eta)$ only  has traps, no particles.}\label{fig:mapfro}
    \label{fig:frozenmap}
\end{figure}

Note that $\xi_k$ is worth $1$ minus the number of consecutive empty sites after $X_k$ in $\eta$, i.e. after the $k$-th particle in the FEP. Because of the exclusion rule in the FEP, there can be at most one particle at each site in $\xi$, so $\xi \in \Gamma_K$. We thus interpret $\xi$ as a SWT configuration on $\T_K$.

\medskip

The application $\eta \mapsto (\xi, X_1)$ is clearly $1$-to-$1$ : the configuration $\xi$ is uniquely defined by $\eta$, and similarly, given $\xi$ and a (suitable) position $X_1$ of the first particle right of the origin, the full configuration $\eta$ can be recovered using \eqref{eq:xieta} and \eqref{eq:etaXi}.  In what follows, for any $1\leq K\leq N$ we will denote by 
\eq{eq:defpi}{\Pi:\{\eta\in\Sigma_N, \; |\eta|=K\} \; \longrightarrow \; \Gamma_K\times \T_N}
the previous mapping, meaning that given a configuration $\eta$ and a tagged occupied site $X_1$ in $\eta$, we define 
$\Pi(\eta)=(\xi, X_1)$, where $\xi$ is defined through \eqref{eq:xieta}. For the sake of brevity, we will abuse our notation, and simply write $\xi=\Pi(\eta)$, ignoring the position of the tagged particle, since we are interested in the transient properties of both processes, that do not depend on the starting point of the mapping.

As illustrated in Figure \ref{fig:ergomap}, an ergodic configuration (see \eqref{eq:EN}) of the FEP is mapped to a SSEP configuration without traps, i.e. an ergodic SWT configuration, meaning
\eqs{\eta\in\mathcal{E}_N^{\FEP}, \; |\eta|=K \quad \Longrightarrow \quad  \xi:=\Pi(\eta)\in \mathcal{E}_K^{\SWT} = \Sigma_K.}
This mapping in the ergodic case was already known, and was successfully exploited in previous works \cite{ayre_mixing_2024, AGLS23}. The mapping to the SWT, however, in the general case, is to our knowledge a novelty of our work.
Similarly, as illustrated in Figure \ref{fig:frozenmap}, a frozen FEP configuration (see \eqref{eq:FN}) is mapped into a configurations with no particles, so a frozen SWT (see \eqref{eq:FK}), meaning 
\eqs{\eta\in\mathcal{F}_N^{\FEP}, \; |\eta|=K \quad \Longrightarrow \quad  \xi:=\Pi(\eta) \in \mathcal{F}_K^{\SWT}.}

\medskip

In what follows, both for $\eta$ and $\xi $, we call a \emph{particle cluster of size $n$} a sequence of $n$ consecutive occupied sites, surrounded on both sides by empty sites or traps. Analogously, we call \emph{empty sites cluster of size $n$} a sequence of $n$ consecutive empty sites in $\eta$, surrounded by two particles. It is straightforward to see that a cluster of $n$ particle in $\xi$ correspond to a  particle cluster of $n+1$ particles in $\eta$. Analogously, a cluster of $n$ empty sites in $\eta$, surrounded by the $k$-th and $k+1$-th particles,  corresponds to site $k$ being a trap of depth $n-1$ in $\xi$, i.e. $\xi_k=-n+1$.

\subsubsection{Dynamical mapping}
\label{sec:dynamicmapping}
We now describe the effect of the FEP dynamics started from $\eta$ on the mapped configuration started from $\Pi(\eta)$. Dynamically speaking, given a cluster of at least $2$ particles in $\eta$, the possible jumps of the cluster's particles are the rightmost particle jumping to the right (resp. the leftmost particles jumping to the left). We will see that this corresponds exactly to a jump of the rightmost particle to the right (resp. the leftmost particle to the left) in the corresponding cluster in $\xi$. 

\smallskip

More precisely, we now account for all possible jumps in the facilitated exclusion process $\eta$, and consider their effect on the mapped configuration $\xi:=\Pi(\eta)$. 
\begin{enumerate}[i)]
\item We first take the case, illustrated in Figure \ref{fig:saut_vide_droit}, in which in $\eta$, the $k$-th particle and $k+1$-th particles are neighbours, and followed by an isolated empty site. When the $k+1$-th particle jumps rightwards  to the neighbouring empty site in  $\eta$, the  particle on site $k$ in $\xi$ jumps to site $k+1$.
\item Next, we assume that the $k$-th particle and $k+1$-th particles in $\eta$ are neighbours, but the following site is empty and is the first of a cluster of $n\geq 2$ empty sites. When the $k+1$-th particle jumps rightwards in $\eta$ to the neighbouring empty site, the particle at site $k$ in $\xi$ jumps towards a trap of depth $n-1$, and fills it up by $1$ : the particle disappears, and the trap is now of depth $n-2$. This is illustrated in Figure \ref{fig:saut_droite_piege}.
\item Third, we consider the case where  the $k$-th particle and $k+1$-th particles in $\eta$ are neighbours, and preceded by an isolated empty site. Then, as illustrated in Figure \ref{fig:saut_vide_gauche}, when the $k$-th particle jumps leftwards in $\eta$ to the neighbouring empty site, the particle on site $k$ in $\xi$ jumps to site $k-1$.
\item Finally, the last case to consider is the one represented in Figure \ref{fig:saut_gauche_piege}, where the $k$-th particle and $k+1$-th particles in $\eta$ are neighbours, and preceded by a cluster of $n\geq 2$ empty sites. Then, when the $k$-th particle jumps leftwards in $\eta$ to the neighbouring empty site, the particle at site $k$ in $\xi$ jumps leftwards towards a trap of depth $n-1$, and fills it up by $1$ as in case ii).
\end{enumerate}
\begin{figure}
    \centering
    \begin{subfigure}[t]{0.45\textwidth}
        \vspace{0pt}
        \includegraphics[width=\textwidth]{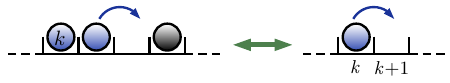}
        \vspace{0pt}
        \caption{Rightward jump to an empty site}\label{fig:saut_vide_droit}
    \end{subfigure}
        \hfill
		\begin{subfigure}[t]{0.45\textwidth} 
		\vspace{0pt}
        \includegraphics[width=\textwidth]{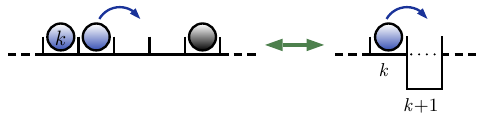}
        \caption{Rightward jump to a trap}\label{fig:saut_droite_piege}
    \end{subfigure}
    
     \begin{subfigure}[t]{0.45\textwidth} 
     \vspace{0pt}
        \includegraphics[width=\textwidth]{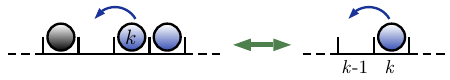}
        \vspace{0pt}
        \caption{Leftward jump to an empty site}\label{fig:saut_vide_gauche}
    \end{subfigure}   
        \hfill
		\begin{subfigure}[t]{0.45\textwidth} 
		 \vspace{0pt}      			
        \includegraphics[width=\textwidth]{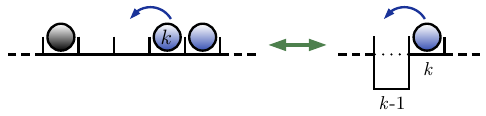}
        \caption{Leftward jump to a trap}\label{fig:saut_gauche_piege}
    \end{subfigure}		
    \caption{Examples of jumps through the mapping. In each subfigure,The FEP is represented on the left and the corresponding SWT on the right. }\label{fig:ex_sauts}
\end{figure}

We can therefore define a dynamical mapping between $\eta(\cdot)$ and the corresponding SSEP with traps. To do so, consider a trajectory $\{\eta(t),\;t\geq 0\}$ of the FEP started from an initial configuration $\eta(0)\in \Sigma_N$ with $K:=|\eta(0)|$ particles. We tag, as in the statical mapping, the first particle at or to the right of the origin. We denote $X_1(0)$ its initial position, and we keep track of its trajectory throughout the process, and denote by $X_1(t)$ its position at time $t$. Since particles cannot cross eachother, we can do the same for all consecutive particles, and define
\eqs{X_1(t)<X_2(t)< \dots <X_K(t)}
as the trajectories of the $K$ particles in the system. As in \eqref{eq:xieta}
we can then define
\begin{equation}
\label{eq:xietat}
\xi_k(t)= 2 + X_k(t) - X_{k+1}(t).
\end{equation}
Given a FEP trajectory $\eta=\{\eta(t), \; t\geq 0\}$, we denote by $\noteblue{\Pi^\star[\eta]:=\{\xi(t), \; t\geq 0\}}$ the trajectory defined by the dynamical mapping \eqref{eq:xietat}. Note that unless at time $t$, the tagged particle is still the first particle at or to the right of the origin, we \emph{do not have} $\Pi(\eta(t))=\Pi^\star[\eta](t)$.
It is straightforward to check that, as represented in Figure \ref{fig:xi}, the process $\{\xi(t), \; t\geq 0\}=\Pi^\star[\eta]$ is a SWT, started from $\xi(0):=\Pi(\eta(0))$, and with generator \eqref{eq:genexi}.

\medskip

Note that in the FEP, the length of a cluster of empty sites can only decrease, because of the kinetic constraint, which corresponds to the fact that in the SSEP with traps, the depth of traps can only decrease.  This construction proves Theorem \ref{thm:mapping}.

\begin{figure}
\centering
\includegraphics[width=13cm]{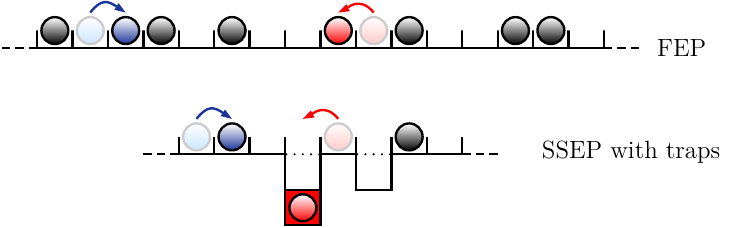}
\caption{Illustration of the dynamical mapping $\eta(t)\mapsto \xi(t)$. A pair of neighbouring particles in the FEP corresponds to a particle in the SSEP with traps. When the right particle of the pair jumps to the right in the FEP, it corresponds to a jump to the right of the particle in the SSEP with traps. Notice the blue particle in the FEP joins a new pair of particles, so in the SSEP with traps the particle doesn't disappear. However, the red particle becomes isolated in the FEP, so in the SSEP with traps it falls into a trap.}
\label{fig:xi}
\end{figure}

\subsubsection{Mixing time for the FEP}
\label{sec:FEPmixing}

Given the sharp estimate on the SWT's mixing time and the previous mapping, it is natural to consider the mixing time of the FEP on the ring. In \cite{ayre_mixing_2024} the latter is studied for a subclass of initial configurations, limited by the estimation of the transience time. Since this limitation is lifted by Theorem \ref{thm:freezetime_cutoff_crit}, we expect that in the general case, the FEP's mixing time on the ring should be obtainable.

\smallskip

It is not hard to show that the mixing time of a FEP on $\T_N$ with $K > N/2$ particles is lower bounded by the mixing time of a SWT on $\T_K$ with $2K-N$ excess particles.
However, upper bounding the mixing time of the FEP by that of the SWT is not straightforward, because our mapping is defined up to the position of a tagged particle (or, symmetrically, up to the current crossing the origin in the SWT). This is not an issue for the transience time, which is not affected by this translation of the system. This is, however, an issue for the mixing time, over which the translation of the system by a random quantity could a priori have an impact. To overcome this issue, and translate the estimation of the SWT's mixing time to the FEP's, one would need to get detailed understanding of the joint distribution of the SWT's configuration and its current through the origin up to time $t$. This interesting question is at this point left for future work.

\smallskip

Still, even though a sharp estimate is not immediately achievable, our refined transience time estimates yield straightforwardly the following result.

\begin{proposition}[Mixing time for the macroscopically supercritical  FEP]
\label{prop:mix_fep}
Let $K(N)$ a sequence such that for all $N$, $K(N) > N/2$ and $\lim\limits_{N \to \infty} K(N)/N = \rho \in (\frac12,1)$. \\
Then, for all $\varepsilon \in (0,1)$, setting $\tau^{\FEP}_{N,K(N)}(\varepsilon)$ the $\varepsilon$-mixing time for the FEP on $\T_N$ with $K(N)$ particles, there exists $C_{\varepsilon, \rho} > 0$ such that
\eqs{\limsup\limits_{N\to\infty} \frac{\tau^{\FEP}_{N,K}(\varepsilon)}{N^2 \log N} \le C_{\varepsilon,\rho}.}
\end{proposition}

This upper-bound generalises \cite[Theorem 2.3(b)]{ayre_mixing_2024}, since we make no assumption on the initial configuration, whereas the result in \cite{ayre_mixing_2024} holds for configurations in which a fixed number of ergodic regions contain enough particles. It can be straightforwardly obtained by combining \cite[Proposition 4.2]{ayre_mixing_2024}, which bounds the log-Sobolev constant of an ergodic FEP on $\T_N$ with $K$ particles, and Remark \ref{rem:fep_diffusif} which upper-bounds the transience time of the FEP.

\section{Proof of Theorem \ref{thm:transiencetimeSSEP}}
\label{sec:proof_tran_swt}
We now prove Theorem \ref{thm:transiencetimeSSEP}, the uniform estimate for the SWT's transience time, that we used in the previous section to prove Theorem \ref{thm:unif_transiencetime}.

\subsection{Labelled SSEP with traps}
\label{subsub:echange}

In order to estimate the transience time for the SWT and prove Theorem \ref{thm:transiencetimeSSEP}, we will use a union bound to consider the trajectory of each individual particle, and estimate the probability that they have explored the whole ring. For Theorem \ref{thm:cutoff_SWT}, we will also consider the trajectory of each particle, to know if the particles have reached the traps. However, this is not very convenient under the exclusion rule, because particles cannot cross each other, so that  any given tagged particle has a heavily constrained and non-diffusive trajectory in general. 
For this reason, we consider instead the following  construction for the SSEP with traps with \emph{distinguishable} particles, in which each particle is labelled, and neighbouring particles exchange their positions at rate $1$. Doing so, individual particles perform (dependent) symmetric random walks, until they reach a trap where they get killed as expected. We now formalise this construction, which can be interpreted as a generalisation of the so-called \emph{interchange process} introduced in \cite{TothInterchange}.

 \medskip

 Fix $\xi\in \Gamma_K$ a configuration of the SSEP with traps, and denote by 
\eqs{
n:=\left|\{k\in \T_K : \xi_k=1\}\right|
}
its number of particles, and give each particle an arbitrary label $j=1,\dots, n$. Let us define $\Xi_j(0)\in\T_K$ as the initial position of the particle labelled $j$. We also define variables $\delta_j(0)=1$, which are initially set to $1$ to indicate that each particle is initially live. We will set instead $\delta_j=0$ whenever particle $j$ is killed by a trap. 
 We denote by \eqs{\hat \xi(0)=(\xi, \Xi_1,\dots,\Xi_n, \delta_1,\dots,\delta_n)(0)\in \Gamma_K\times \T_K^n\times \{0,1\}^n.}
 this initial labelled SWT configuration, and define a Markov process 
\eqs{
\hat \xi(t)=(\xi, \Xi_1,\dots,\Xi_n, \delta_1,\dots,\delta_n)(t)
}
whose evolution we now describe.

\medskip

In $ \hat \xi$, $\xi$ is a SSEP with traps, that we assume defined by a set of independent Poisson clocks $\{\mathscr{T}_{k}: k\in \T_k\}$ ringing at rate $1$, each describing the times at which the edge $(k,k+1)$ is activated. Each variable  $\noteblue{\Xi_j(t)}$ represents the position at time $t$ of the particle labelled $j$, and $\noteblue{\delta_j(t)}$ is equal to $1$ if the particle labelled $j$ is live at time $t$, and $0$ otherwise. The process $ \hat \xi$'s time evolution is then determined by these Poisson clocks, with the following rules; when  the clock $\mathscr{T}_{k}$ rings, the following happens depending on the value of $\xi$ at sites $k$ and $k+1$;
\begin{itemize}
\item {[\emph{Two frozen sites}]} if both $\xi_{k}\leq 0$ and $\xi_{k+1}\leq 0$, nothing happens.
\item {[\emph{Particle jump to a trap}]} If $\xi_{k} = 1$ and $\xi_{k+1}< 0$, the particle at site $k$, labelled $j$, jumps to the trap at site $k+1$ and remains dead there forever, keeping its label. More precisely, $\Xi_j$ jumps from $k$ to $k+1$ and $\delta_j$ from $1$ to $0$, and both keep those new values forever. The trap loses $1$ depth, so that $\xi_k$ changes from $1$ to $0$ and $\xi_{k+1}$ increases by $1$. The same construction holds if instead $\xi_{k} <0$ and $\xi_{k+1}=1$, but with opposite  roles  for $k$ and $k+1$.
\item {[\emph{Particle jump to an empty site}]} If $\xi_{k} = 1$ and $\xi_{k+1} = 0$, the particle at site $k$, labelled $j$, jumps to $k+1$, once again keeping its label but this time remaining live, so that $\delta_j$ remains unchanged, $\Xi_j$ jumps from $k$ to $k+1$, and $\xi_k$ and $\xi_{k+1}$ exchange their values. As before, is instead $\xi_{k} = 0$ and $\xi_{k+1} = 1$, the same construction holds with opposite roles for $k$ and $k+1$.
\item {[\emph{Two particles swap positions}]} Finally, if $\xi_{k} = \xi_{k+1} = 1$, the particle at site $k$, labelled $j$ and the one at $k+1$, labelled $j'$ exchange their positions, meaning that $\xi$ remains unchanged, $\Xi_j$ jumps from $k$ to $k+1$ and $\Xi_{j'}$ from $k+1 $ to $k$.
\end{itemize}

\medskip

It is not hard to check that the following affirmations are guaranteed by this construction:
\begin{enumerate}[i)]
\item if $\delta_j(t)=\delta_{j'}(t)=1$, then $\Xi_j(t)\neq \Xi_{j'}(t)$; in other words, two live particles cannot coexist on the same site.
\item The sets $\{k\in \T_K :\;\xi_k(t)=1\}$ and $\{\Xi_j(t),1\leq  j\leq n :  \delta_j(t)=1\}$ are identical: the set of occupied sites  in $\xi$ is at all time the set of live labelled particles.
\item Trivially, the value of $\xi_k$ at time $t$ is given by the initial depth of its trap (which can be $0$), plus the number of labelled particles which are currently there, which translates as the identity 
\eqs{\xi_k(t)=\xi_k(0){\bf 1}_{\{\xi_k(0)\leq 0\}} +\sum_{j=1}^n {\bf 1}_{\{\Xi_j(t)=k\}}.}
\end{enumerate}
Forgetting in $\hat \xi(t)$ the labels $\Xi_j(t)$'s and the $\delta_j$'s, the resulting process $\xi(t)$ is clearly a SSEP with traps, driven by the generator $\mathscr{L}^{\SWT}_K$ defined in \eqref{eq:genexi}. The upside of this construction is that, for any $1\leq j\leq n$, the trajectory $\{\Xi_j(t), \;t\geq 0\}$ of the particle labelled $j$ is a random walk, jumping at rate $1$ to either of its neighbours, and interrupted only when it falls in a trap of positive depth. Of course, those trajectories are not independent from one another, because two of them cannot be at the same time at the same site. Nevertheless, they behave as random walkers, which is what we needed.
\bigskip

\subsection{Proof of Theorem \ref{thm:transiencetimeSSEP}}
\label{sec:proof_th_3}

We now have the tools we need to prove Theorem \ref{thm:transiencetimeSSEP}. Recall that we wish to bound  from above the probability for a SSEP with traps to be transient.

\medskip

We first couple a labelled SWT $\hat \xi(t)=(\xi, \Xi_1,\dots,\Xi_n, \delta_1,\dots,\delta_n)(t)$ with a labelled SSEP $\hat\sigma (\cdot)$, by removing all the traps from the initial state, but using the same Poisson clocks $\mathscr{T}_{k}$ and the same initial labels for both processes. More precisely, fix an initial SWT configuration $\xi$, and  define an initial SSEP configuration $\sigma\in \Sigma_K$ by $\sigma_k=0\vee \xi_k$ for $k\in \T_K$. Choose the same initial labels $(Y_j)$ in $\sigma$ and $\xi$, meaning  $Y_j(0)=\Xi_j(0)$ for  $j\in \{1,\dots, n\}$. Since there are no traps in $\sigma$, we have no need, in $\hat \sigma$, for the auxiliary variables $\delta_j(t)$ which remain constant equal to $1$, and we can proceed as in Section \ref{subsub:echange} to build a labelled  SSEP
\eq{eq:Defsigmahat}{
\hat \sigma(t)=(\sigma, Y_1,\dots,Y_n)(t)\in \Gamma_K\times \T_K^n.
}
started from $(\sigma, Y_1,\dots,Y_n)(0)$, where $\{Y_j(t), \; t\geq 0\}$ represents the trajectory in $\T_K$ of the particle labelled $j$. Note that the trajectories $(Y_1, ... Y_n)(\cdot)$ form an \textit{interchange process}, studied in \cite{TothInterchange,HS21} and used in \cite{morris_mixing_2006,tran_cutoff_2022}.

Define the moment $\tau_j$ at which the particle labelled $j$ falls into a trap, namely 
\eqs{\tau_j:=\inf\{t\geq 0 : \delta_j(t)=0\}.}
Since the SSEP with traps is attractive (see Section \ref{subsec:att_swt}), $\xi(t)\leq \sigma(t)$ for any $t\geq 0$, and the only discrepancies between the two processes are particles that fell in a trap in $\hat\xi$. Indeed, since we used the same labels and the same clocks, the only way for particle $j$ to be in different positions in $\sigma(t)$ and $\xi(t)$ is to have fallen into a trap, therefore for any $t$, any $j$,
\eq{eq:idXiY}{
\Xi_j(t) = Y_j(t \wedge \tau_j) .
}

In particular, we claim that given a trajectory $\hat \xi$ of the labelled SSEP with traps, and the associated SSEP $\hat \sigma$,
\eq{eq:nottransient}{
\xi(t)\in \mathcal{T}^{\SWT}_K \qquad \Longrightarrow \qquad  \exists j: \; \{Y_j(s),\; s \le t\} \neq \T_K.
}
Indeed, assume that $\xi$ is in a transient state at time $t$, there must exist, up to time $t$, at least one live particle, labelled $j$, and one trap with positive depth at some site $k\in\T_K$. This means that the partial trajectory $\{\Xi_j(s), \;s\leq t\}$ does not contain $k$, otherwise  particle $j$ would have fallen in the trap (which has positive depth since time $0$) at its first visit at site $k$. But since the particle is still live, according to \eqref{eq:idXiY}, we must have $t\leq \tau_j$, so that
\eqs{k\notin \{\Xi_j(s) : s \le t\} =\{Y_j(s) : s \le t\},}
which proves \eqref{eq:nottransient}.  For simplicity, we still denote by ${\bf P}_\sigma^K$ the distribution of the labelled SSEP on $\T_K$ started from the configuration $\sigma$ with arbitrary initial labels, thanks to \eqref{eq:nottransient}, to prove Theorem \ref{thm:transiencetimeSSEP} it is enough to show that  
\eq{eq:QKnoneploration}{\lim_{t\to\infty}\;\sup_{K\geq 0}\; \sup_{\sigma\in \Sigma_K}{\bf P}_\sigma^K\big(\exists j, \{Y_j(s) : s \le t K^2\log K\} \neq \T_K)=0.}

\medskip

 Let now $\sigma \in \Sigma_K$ be an initial configuration with $n\le K$ particles. Further denote by $\noteblue{\P}$ the distribution of a rate $1$ random walk $\{Y(t),\; t\geq 0\}$ on $ \T_K$ started from the origin ($Y(0)=0$). Then, by union bound, since $n\leq K$, 
\eq{eq:QTk}{{\bf P}_\sigma^K\big( \exists j\leq n, \{Y_j(s),\; s\leq tK^2\log K\}\neq \T_K\big)\leq K\P\big(\{Y(s),\; s\leq tK^2\log K\}\neq \T_K\big).}
By a standard estimate on random walks, given in  Appendix \ref{app:proof_RW} for completeness, there exists a positive constant $c$ such that applying Corollary \ref{cor:1} at time $t K^2 \log K$,
\eq{eq:explorationtime}{\PP(\{Y(s),\; s\leq tK^2\log K\}\neq \T_K) \le c K^{-t}.}
Injecting this bound into \eqref{eq:QTk} proves \eqref{eq:QKnoneploration} and concludes the proof of Theorem \ref{thm:transiencetimeSSEP}.

\subsection{Mixing time for the SWT}
\label{sec:proof_cutoff_mix}

In this Section, we show Theorem \ref{thm:cutoff_mix} and also derive specific bounds on the SWT mixing time for different numbers of excess particles, which proves cutoff in some regimes. To obtain Theorem \ref{thm:cutoff_mix}, we study the mixing time according to the number of excess particles. For $s > 0$, recall that $\pi_{K,s}$ is the invariant law of a SWT on $\T_K$ with $s$ excess particles. For $s=0$, we set $\pi_{K,s}$ the Dirac measure on the empty configuration. We can now define the $\varepsilon-$mixing time for the SWT configurations with $s$ excess particles:
\eq{eq:def_tmixs}{\tau_{K,s}^{\textsc{swt}}(\varepsilon) = \inf\Big\{t \ge 0 : \forall \xi \in \Gamma_{K},\; S(\xi)=s,\; d_{\textsc{tv}}(\bP^K_{\xi}(\xi(t) \in \cdot), \pi_{K,s}) \le \varepsilon\Big\}.}

Recall that a SSEP with traps with $s \ge 0$ behaves in two phases: first the transient phase, then once in the ergodic set of configurations it behaves like a SSEP with $s$ particles. So, denoting by 
\eqs{\tau^{\textsc{ssep}}_{K,s}(\varepsilon) = \inf\Big\{t \ge 0 : \forall \sigma \in \Sigma_{K},\; |\sigma|=s,\; d_{\textsc{tv}}(\bP_{\sigma}^K(\sigma(t) \in \cdot), \pi_{K,s}) \le \varepsilon\Big\}}
the $\varepsilon-$mixing time of the SSEP on $\T_K$ with $s$ particles, it follows immediately that

\eq{eq:t_mix_tr}{\theta_{K,s}^{\SWT}(\varepsilon) \vee \tau^{\textsc{ssep}}_{K,s}(\varepsilon) \le \tau^{\SWT}_{K,s}(\varepsilon) \le \theta^{\SWT}_{K,s}(\varepsilon/2) + \tau^{\textsc{ssep}}_{K,s}(\varepsilon/2).}

The lower bound is quite straightforward, and the upper bound comes with applying the Markov property at time $\theta^{\SWT}_{K,s}(\varepsilon/2)$. 

\paragraph{Estimates on the transience and mixing times.} We show in Lemma \ref{lem:t_tran_critique} the following estimates for the \emph{transience time}:

\eq{eq:encadr_swt}{\vartheta_{K,s}^{\SWT}(\varepsilon) \le \theta_{K,s}^{\SWT}(\varepsilon) \le \Theta_{K,s}^{\SWT}(\varepsilon),}
where the bounds are given by 
\eq{eq:encadr_swt2}{\vartheta_{K,s}^{\SWT} (\varepsilon) = \frac{1}{\pi^2}K^2\log\frac{K}{s \vee 1} + \mathcal{O}_{\varepsilon}(K^2),}
\eq{eq:encadr_swt3}{\Theta_{K,s}^{\SWT} (\varepsilon) = \frac{1}{\pi^2}K^2\log\frac{K}{s \vee 1} + \mathcal{O}_{\varepsilon} (K^2 \log\log K),}
where $\mathcal{O}_\varepsilon$ means the constant depends on $\varepsilon$. In the case where $s$ is of order $\mathcal{O}(K)$, we actually have $\theta_{K,s}^{\SWT} (\varepsilon)= \mathcal{O}_{\varepsilon}(K^2)$, see Remark \ref{rem:sctransience}.

In order to prove Theorem \ref{thm:cutoff_mix}, according to \eqref{eq:t_mix_tr}, we will use the following upper bound on the SSEP's mixing time, that we show in Appendix \ref{app:tmix}:
\begin{proposition}[Upper bound on the mixing time of the SSEP] For all $K \ge 1$ and $0\le s \le K$,
\eq{eq:ssep_pas_sharp}{\tau_{K,s}^{\textsc{ssep}}(\varepsilon) \le \frac{K^2}{2\pi^2} \left(\log \frac{s \wedge (K-s)}{\varepsilon} + \log 4 /\pi +  o(1) \right).}
\end{proposition}

\begin{proof}[Proof of Theorem \ref{thm:cutoff_mix}. ]
To obtain a uniform mixing time, over all values of $s$, we choose for all $K$ a number of particles $s^\star_K$ such that
\eq{eq:def_tmix_s_max}{\tau^{\SWT}_{K,s^\star_K}(\varepsilon) = \max \Big\{\tau_{K,s}^{\SWT}(\varepsilon), 0 \le s \le K\Big\}=\tau^{\SWT}_{K}(\varepsilon),}
where the former was defined in \eqref{eq:def_tmixs} and the latter is the overall SWT mixing time defined in \eqref{eq:def_tmix}.

Clearly, $\tau_{K}^{\SWT}(\varepsilon)\geq \tau_{K,0}^{\SWT}(\varepsilon)= \theta_{K}^{\SWT}(\varepsilon)\ge t_K^\star + \mathcal{O}_{\varepsilon}(K^2) $ according to \eqref{eq:encadr_swt2}. \\
For the upper bound, using \eqref{eq:ssep_pas_sharp}, we have $\tau_K^{\SWT}(\varepsilon) \le t_K^{\star} \left(1 - \frac{1}{2}\frac{\log s_K}{\log K}\right) + \mathcal{O}_{\varepsilon}(K^2 \log \log K)$, which is less than $t_K^{\star} + \mathcal{O}_{\varepsilon}(K^2 \log \log K) $. This yields \eqref{eq:tmix_fenetre} and the cutoff.
\end{proof}

\begin{table}
\bgroup
\def\arraystretch{1.7}
\begin{tabular}{|c|c|c|c|}
\hline
$s$ & $\theta^{\SWT}_{K,s}(\varepsilon)$ & $\tau^{\textsc{ssep}}_{K,s}(\varepsilon)$ & Cutoff? \\
\hline
0 & $\frac{1}{\pi^2}K^2 \log K  + \mathcal{O}_{\varepsilon}(K^2 \log\log K)$  & 0 & Yes \\
\hline
$\mathcal{O}(1)$ & $\frac{1}{\pi^2}K^2 \log K/s  + \mathcal{O}_{\varepsilon}(K^2 \log\log K)$ & $\mathcal{O}_{\varepsilon,s}(K^2)$  & Yes \\
\hline 
$s = K^{\alpha} < \frac{K}{2}$  & $\frac{1-\alpha}{\pi^2}K^2 \log K  + \mathcal{O}_{\varepsilon}(K^2 \log\log K)$ & $\frac{\alpha}{8\pi^2}K^2 \log K + \mathcal{O}_{\varepsilon}(K^2) $ & ? \\
\hline
$s = \delta K < \frac{K}{2}$  & $\mathcal{O}_{\varepsilon}(K^2)$ & $\frac{1}{8\pi^2}K^2 \log (\delta K)+ \mathcal{O}_{\varepsilon}(K^2) $ & Yes \\
\hline
$s = \delta K > \frac{K}{2}$  & $\mathcal{O}_{\varepsilon}(K^2) $ & $ \frac{1}{8\pi^2}K^2 \log ((1-\delta) K) + \mathcal{O}_{\varepsilon}(K^2) $ & Yes \\
\hline
$s = K - K^\alpha$  & $\mathcal{O}_{\varepsilon}(K^2) $ & $\frac{\alpha}{8\pi^2}K^2 \log (K) + \mathcal{O}_{\varepsilon}(K^2) $ & Yes \\
\hline
$s =  K - \mathcal{O}(1)$  & $\mathcal{O}_{\varepsilon}(K^2) $ & $ \mathcal{O}_{\varepsilon,s}(K^2) $ & ? \\
\hline
\end{tabular}
\egroup
\caption{Behaviour of the mixing time of the SWT on $\T_K$ with $s$ excess particles.}
\label{tab:tab_swt}
\end{table}

\paragraph{Mixing time for different values of $s$. }
We are now able to give precise estimates on the mixing time of the SWT with $s$ excess particles depending on the value of $s$. For this, we need precise estimates on the SSEP mixing time from \cite{lacoin_simple_2017}, given by the following Proposition.
\begin{proposition}[Bounds on the mixing time of the SSEP \cite{lacoin_simple_2017}] If $s_K\wedge (K-s_K) \to +\infty$, 
\eq{eq:encadr_ssep}{\frac{1}{8\pi^2}K^2\log (s_K \wedge (K-s_K)) + \mathcal{O}_{\varepsilon}(K^2) \le \tau_{K,s_K}^{\textsc{ssep}}(\varepsilon).}
\eq{eq:upper_bound_ssep}{\tau^{ \textsc{ssep}}_{K,s_K}(\varepsilon) \le \frac{1}{8\pi^2} K^2 \log (s_K\wedge (K-s_K)) + \mathcal{O}_{\varepsilon}(K^2),} 
\end{proposition}
The lower bound comes from \cite[Equation (2.1)]{lacoin_simple_2017} and is in fact true for all $s$, and the upper bound is obtained from \cite[Propositions 3.2, 3.3 and Equation (3.13)]{lacoin_simple_2017}.
Hence, denoting by $\Theta_{K,s}^{\SSEP} (\varepsilon)$ the adequate upper-bound for $\tau_{K,s}^{\textsc{ssep}}(\varepsilon)$, given either by \eqref{eq:upper_bound_ssep} or \eqref{eq:ssep_pas_sharp} depending on the value of $s$, and $\vartheta_{K,s}^{\textsc{ssep}}(\varepsilon)$ the lower bound in \eqref{eq:encadr_ssep}, \eqref{eq:t_mix_tr} becomes
\eq{eq:t_mix_tr_2}{\vartheta_{K,s}^{\SWT}(\varepsilon) \vee 
\vartheta_{K,s}^{\SSEP}(\varepsilon) \le \tau^{\SWT}_{K,s}(\varepsilon) \le \Theta_{K,s}^{\SWT} (\varepsilon/2) + \Theta_{K,s}^{\SSEP} (\varepsilon/2),}
which directly yields the following bounds on the SWT mixing time according to the number of excess particles $s$.
\begin{proposition}[Bounds on the mixing time of the SWT in different regimes] \label{prop:t_mix_swt_all_s} For all sequence $(s_K)$ such that $0 \le s_K \le K$, setting $s_K^{\#} = s_K \wedge (K-s_K)$, then the mixing time of the SWT with $s_K$ excess particles has the following behaviour.
\begin{itemize}
    \item If $s_K^{\#} $ goes to infinity with $K$, 
    \eq{eq:tmix_swt_big_s}{\frac{K^2}{\pi^2}\left(\log \frac{K}{s_K} \vee \frac{\log s_K^{\#}}{8}  \right) + \mathcal{O}_{\varepsilon}(K^2) \le \tau^{\SWT}_K(\varepsilon) \le \frac{K^2}{\pi^2} \left( \log \frac{K}{s_K} + \frac{\log s_K^{\#}}{8} \right) +\mathcal{O}_{\varepsilon}(K^2\log\log K)}
    \item In all cases,
    \eq{eq:tmix_swt_small_s}{\frac{K^2}{\pi^2} \log \frac{K}{s_K} + \mathcal{O}_{\varepsilon}(K^2) \le \tau^{\SWT}_K(\varepsilon) \le \frac{K^2}{\pi^2}\left( \log \frac{K}{s_K} + \frac{1}{2} \log s_K^{\#}\right) +\mathcal{O}_{\varepsilon}(K^2\log\log K)}
\end{itemize}
\end{proposition}
These bounds allows to estimate $\tau_{K,s}^{\SWT}$ quite precisely, and obtain cutoff for some behaviours of $(s_K)$, see Table \ref{tab:tab_swt} for a summary. The SWT transience time exhibits cutoff for small $s$ and the SSEP mixing time exhibits cutoff for big $s$, and the transience time and the SSEP mixing time have roughly opposed monotonicities in $s$, and balance each other out. More precisely, when $s$ is small, the transience time dominates and the upper and lower bound in \eqref{eq:tmix_swt_small_s} are both of order $t_K^{\star}$; when $s$ is big enough, the SSEP mixing time dominates and the upper and lower bound in \eqref{eq:tmix_swt_big_s} are both of order $\tau^{\textsc{ssep}}_{K,s}$.

\section{Proof of Theorem \ref{thm:cutoff_SWT}}
\label{sec:cutoffSWT}

\subsection{Critical transience time with a unique trap}
\label{sec:uniquetrap}

Notice that traps act as empty equilibrium reservoirs on neighbouring particles (meaning those get removed at rate $1$), until they are filled up, at which point they no longer have any effect. For a critical configuration, the moment when all traps have filled up coincides with the moment when the configuration is completely empty, which makes estimates simpler. For this reason, we start by considering the simpler case, represented in Figure \ref{fig:unique}, where the initial configuration is critical and has a unique trap. We obtain in Lemma \ref{lem:unique_well} below a sharp estimate for the SSEP in contact at both extremities with empty equilibrium reservoirs, which directly translates as an estimate for the transience time of a critical single-trap SWT.

\begin{figure} 
\centering
\includegraphics[width=0.7\textwidth]{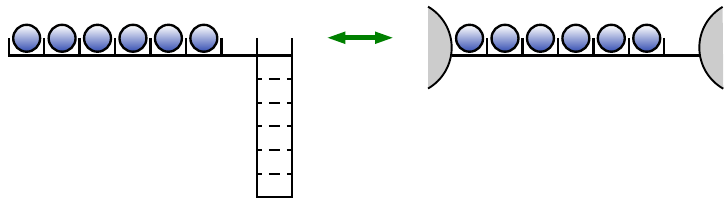}

\caption{On the left, a critical SSEP with traps with only one trap. On the right, a corresponding SSEP with empty reservoirs.}
\label{fig:unique}
\end{figure}

\medskip

Let $\xi\in \mathcal{T}^{\SWT}_{\star,K}$ be a critical SWT configuration, so that starting from $\xi$, $\xi(\cdot)$ stops being transient when it is identically $0$. By attractiveness, we can further assume without loss of generality that every site except the trap is initially occupied, since we want to estimate the largest transience probability at time $t$, and that removing particles will only lower the transience time. The position of the trap does not matter, therefore we consider a SWT $\xi(t)$ with initial configuration 
\eqs{\xi_k^\star={\un}_{\{1\leq k\leq K-1\}}-(K-1){\un}_{\{k=K\}}.}

Since there are exactly $K-1$ particles in the system $\xi(t)$'s particles behave as if the trap at site $K$ was infinite, which breaks the system's periodicity. In particular, defining $\noteblue{\Lambda_K:=\llbracket 1, K-1\rrbracket}$, as represented in Figure \ref{fig:unique}, the restricted process $\sigma(t):=\xi_{\mid \Lambda_K}(t)$ behaves like a regular SSEP on $\Lambda_K$ in contact at both extremities with empty reservoirs, meaning particles are deleted at rate $1$ at sites $1$ and $K-1$. We will strongly rely on the work of \cite{tran_cutoff_2022} about cutoff of the mixing time for the SSEP with reservoirs, and adapt it to the simpler case of empty reservoirs in the following. In what follows, we denote by $\noteblue{\Q^K_\sigma}$ the law of a SSEP  in contact at both extremities with empty reservoirs started from $\sigma\in \{0,1\}^{\Lambda_K},$ and by $\noteblue{\E^K_\sigma}$ the corresponding expectation. We choose this different notation for the law to emphasise that the dynamics takes place on the slightly smaller and \emph{non-periodic} set $\Lambda_K$. This process's invariant measure is the Dirac measure on the empty configuration on $\Lambda_K$ noted $\delta_{0,K}$, therefore for any initial configuration $\sigma$, and any time $t$, \eqs{d_{TV}(\Q_\sigma^K(\sigma(t) \in \cdot ) , \delta_{0,K}) = \Q_\sigma^K(|\sigma(t)|>0),} 
where the left-hand side is the total variation distance between the distribution of the SSEP at time $t$ and its invariant state, whereas the right-hand side is the probability that at time $t$, there is still at least one particle in $\sigma(t)$.
With this identity, we can therefore define for all $\varepsilon \in (0,1)$ the $\varepsilon$-\emph{mixing time}
\begin{align}
\theta^{\textsc{mix}}_K (\varepsilon) &= \inf\left\{t \ge 0 : \forall \sigma \in \{0,1\}^{\Lambda_K} , \Q_\sigma^K(|\sigma(t)| > 0) \le \varepsilon\right\}\\
&=\inf\left\{t \ge 0 : \Q_{\un}^K(|\sigma(t)| > 0) \le \varepsilon\right\}\\
&=\inf\left\{t \ge 0 : {\bf P}_{\xi^\star}^K(\xi(t)\in \mathcal{T}^{\SWT}_K) \le \varepsilon\right\},
\end{align}
where  $\un\in \{0,1\}^{\Lambda_K}$ is the full configuration $\noteblue{\un\equiv 1}$. The second identity holds by the SSEP's attractiveness, whereas the third holds because $\xi$ is transient iff $\sigma$ has at least one particle. Define
\begin{equation}
\label{eq:deftau}
\tau_K^\star = \inf \{t \ge 0 : \E_\un^K \left[|\sigma(t)| \right] \le 1\}.
\end{equation}
We first link $\tau_K^\star$ and the $\varepsilon$-mixing time.
\begin{lemma} 
\label{lem:unique_well}
There exists a constant $C>0$ such that for all $\varepsilon \in (0,1)$,
\begin{equation}
\label{eq:framing1}
 \tau_K^\star - C K^2\left(1 + \log \left(\frac{1}{1-\varepsilon}\right) \right) \le \theta^{\textsc{mix}}_K (\varepsilon) \le \tau_K^\star + C K^2 \log \frac{1}{\varepsilon}.
\end{equation}
Furthermore, recalling the definition \eqref{eq:tNstar} of $t_K^\star$, we have
\begin{equation} \label{eq:t_K}
|\tau_K^\star -t_K^\star|=\mathcal{O}(K^2).
\end{equation}
\end{lemma}
The proof of Equation \eqref{eq:t_K} can be straightforwardly adapted from  \cite[Section 5]{tran_cutoff_2022} where a slightly different  definition of $\tau_K^\star$ is considered, therefore we do not reproduce the proof here. To prove \eqref{eq:framing1}, we use the following results.
\begin{lemma}(\cite[Proposition 6]{tran_cutoff_2022})
\label{prop:ND}
The SSEP with reservoirs preserves negative dependence, in particular, for any initial configuration $\sigma$,
for all $t \ge 0$, $Var_\sigma^K(|\sigma(t)|)\le \E_\sigma^K[|\sigma(t)|] $.
\end{lemma}
We will also use a classical estimate on the hitting time of random walks, taken from the same article.
\begin{lemma}(\cite[Lemmas 8 and 12]{tran_cutoff_2022})
\label{lem:exp1}
Let $(Y(t))_{t \ge 0}$ a continuous time simple symmetric random walk on $\{0,...,K\}$. Let $T_0 = \inf \{t \ge 0 : Y(t) = 0\}$.
Then there exists a constant $c>0$ independent of K such that for all $0 \le i \le K$,
\begin{equation}
\PP_i(T_0 \ge 2 K^2) < e^{-c}. 
\end{equation} 
In particular, for any initial condition $\sigma$ on $\Lambda_K$,
\begin{equation}
\label{lem:exp2}
\E_{\sigma}^K[|\sigma(2 t)|] \le e^{-c \ent{\frac{t}{K^2} }} |\sigma|.
\end{equation}
\end{lemma}
We will not prove these two results, and refer to the original article for the proofs. The first claim in Lemma \ref{lem:exp1} is \cite[Lemma 8]{tran_cutoff_2022}, whereas the second is \cite[Lemma 12]{tran_cutoff_2022}, and is a consequence of a  union bound over all particles initially in the system, and Markov's property applied to the sub-trajectories $\{Y(t), t\in [2jK^2,2(j+1)K^2]\}$, $j\geq 0$ of each particle.

\begin{proof}[Proof of Lemma \ref{lem:unique_well}]We now use the two previous Lemmas to prove our estimate on the mixing time of the SSEP with empty reservoirs, following  the same steps as  \cite[Theorem 2]{tran_cutoff_2022}. 
\paragraph{Upper bound of \eqref{eq:framing1}: }
By Markov inequality, we start by writing 
\eqs{\Q_{\un}^K(|\sigma(t)| \ge 1) \le \E_{\un}^K [|\sigma(t)|] .}
Now we choose $t = \tau_K^\star + 2 K^2 m$ with $m=\frac{1}{c}\log \left(\frac{1}{\varepsilon}\right)$, $c$ being the constant introduced in Lemma \ref{lem:exp1}. We write by Markov property, \eqref{lem:exp2}, and the definition of $\tau_K^\star$,
\begin{align}
\E_{\un}^K[|\sigma(\tau_K^\star + 2K^2 m)|] &= \E_{\un}^K \left[\E^K_{\sigma(\tau_K^\star)} [|\sigma(2K^2 m)|]\right] \label{eq:markov}\\
&\le e^{-cm} \E_{\un}^K[|\sigma(\tau_K^\star)|] \nonumber\\
&\le \varepsilon. \nonumber
\end{align}
In particular, for any $t \ge \tau_K^\star +  \frac{2}{c} K^2 \log\left(\frac{1}{\varepsilon}\right)$, $d_{TV}(\Q_{\un}^K(\sigma(t) \in \cdot), \delta_{0,K}) \le \varepsilon$, therefore 
\begin{equation}
\theta^\textsc{mix}_K(\varepsilon) \le \tau_K^\star +  \frac{2}{c} K^2 \log\frac{1}{\varepsilon}.
\end{equation}

\paragraph{Lower bound of \eqref{eq:framing1}: }
Recall that by Lemma \ref{prop:ND}, $Var_{\un}^K(|\sigma(t)|) \le \E_{\un}^K[|\sigma(t)|]$ at all times.
We also have, by the second moment method, and Jensen's inequality
\begin{equation}
\Q_{\un}^K(|\sigma(t)| > 0)\ge \frac{\E_{\un}^K[|\sigma(t)|]^2}{\E_{\un}^K[|\sigma(t)|^2]} \ge 1 - \frac{Var_{\un}^K(|\sigma(t)|)}{\E_{\un}^K[|\sigma(t)|]^2}. 
\end{equation}
so that
\begin{equation}
\label{eq:dTV}
d_{TV}(\Q_{\un}^K(\sigma(t) \in \cdot), \delta_{0,K}) \ge 1 - \frac{1}{\E_{\un}^K[|\sigma(t)|]}.
\end{equation}
By the same reasoning as in equation \eqref{eq:markov}, for $m \in \N$, $\E_{\un}^K[|\sigma(\tau_K^\star)|] \le e^{-cm} \E_{\un}^K [|\sigma(\tau_K^\star - 2K^2 m)|]$, which gives the lower bound 
\begin{equation}
\E_{\un}^K[|\sigma(\tau_K^\star - 2K^2m)|] \ge e^{cm} \E_{\un}^K[|\sigma(\tau_K^\star)|] = e^{cm}.
\end{equation}

By choosing $m:= \frac{1}{c} \log\left(\frac{1}{1-\varepsilon}\right) +1> \frac{1}{c} \log\left(\frac{1}{1-\varepsilon}\right)$, we obtain using \eqref{eq:dTV} 
\begin{equation}
d_{TV}(\Q_{\un}^K(\sigma(\tau_K^\star - 2K^2m) \in \cdot), \delta_{0,K}) > \varepsilon,
\end{equation}
so that, $\theta_K^{\textsc{mix}}(\varepsilon) \ge \tau_K^\star - \frac{2}{c} K^2\log \left(\frac{1}{1-\varepsilon}\right) - 2 K^2$ which gives us the lower bound. Setting $C = \frac{2}{c} \vee 2$ concludes the proof of \eqref{eq:framing1}. 

\end{proof}

\subsection{Generalisation to all  initial configurations}
\label{subsub:all_critical}

Now that we have solved the case of a critical configuration with a single trap, we wish to consider a general configuration and get back to the single trap case. Because of attractiveness, solving the critical case immediately yields estimates  for the supercritical and subcritical  transience time. More precisely, given a SWT configuration $\xi\in \Gamma_K$, and recalling the definition of the number of excess particles $S(\xi)$ (see Section \ref{sec:defxi}), define 
\eqs{\mathcal{T}^{\SWT}_{\star,K} = \big\{\xi \in \mathcal{T}_K^{\SWT} : S(\xi) = 0\big\}}
the set of transient critical configurations of the SWT on $\T_K$. 
Recall from \eqref{eq:critprobSWT} the definition of the maximal critical transience probability  $p_{\star,K}^{\SWT}(t)$, and define the associated $\varepsilon$-transience time 
\eqs{
\theta^{\SWT}_{\star,K}(\varepsilon) = \inf \{t \ge 0 : p_{\star,K}^{\SWT}(t) \le \varepsilon\}.
}
We claim the following;
\begin{lemma}
\label{lem:t_tran_K}
The maximal transience time for the SWT is the critical one, meaning
\begin{equation}
\label{eq:critSWT}
\theta^{\SWT}_K(\varepsilon) = \theta^{\SWT}_{\star,K}(\varepsilon).
\end{equation}
\end{lemma}
\begin{proof}
The proof is straightforward thanks to the attractiveness of the SSEP with traps. Given a supercritical configuration $\xi\in \Sigma_K$, we associate it with a critical configuration $ \xi^\star\leq \xi$ e.g. by removing particles in $\xi$ and keeping the same traps. Then, under the basic coupling introduced in Section  \ref{subsec:att_swt}, at any time $t$, we have $\xi^\star(t)\leq \xi(t)$, so that if $\xi$ is transient at time $t$, then $\xi^\star$ is as well. In particular, $p_{\xi^\star}(t)\geq p_\xi(t)$. The exact same argument works for a subcritical configuration $\xi$, which we couple with a critical one $\xi^\star$, for example by adding particles. In the same way, if $\xi$ is still transient at time $t$, then so is $\xi^\star$. Putting these two cases together yields as wanted 
\eqs{p_{K}^{\SWT}(t)=p_{\star,K}^{\SWT}(t),}
which proves the Lemma.
\end{proof}

Because of this Lemma, we now only need to estimate the  critical  transience time. However, we present here a more detailed estimate for supercritical configurations depending on the value of $S(\xi)$, which will be of use later on.
For this purpose, define for $s\in \llbracket 0,K\rrbracket$
\eqs{p^{\SWT}_{K,s} = \sup\big\{p_\xi(t): S(\xi) = s\big\} \qquad \mbox{ and }\qquad \theta^{\SWT}_{K,s}(\varepsilon) = \inf \{t \ge 0 : p_{K,s}^{\SWT}(t) \le \varepsilon\},}
which are the transience probability and transience time for the SWT with $s$ excess particles. 

\begin{lemma} 
\label{lem:t_tran_critique}
There exist $C>0$  such that for all $K \in \Z$, for all $0<\varepsilon<1$,
\begin{equation} \label{eq:encadrement_theta}
\vartheta_{K,s}(\varepsilon) \le \theta^{\SWT}_{K,s}(\varepsilon) \le\Theta_{K,s}(\varepsilon)
\end{equation}
where $t_K^\star$ was defined in \eqref{eq:tNstar}, and 
\begin{subequations}
\label{def:thetaK12}
\begin{align}
\vartheta_{K,s}(\varepsilon)&:=\frac{K^2}{\pi^2}\log\frac{K}{s \vee 1} - C K^2 \left(1 + \log\frac{1}{1-\varepsilon}\right) \label{def:thetaK1}\\ \Theta_{K,s}(\varepsilon)&:= \frac{K^2}{\pi^2}\log\frac{K}{s \vee 1} + C K^2 \left(1 + \log\frac{3 \log K}{\varepsilon}  \right). \label{def:thetaK2}
\end{align}
\end{subequations}
Note that for $s=0$ (critical case), the leading term in both $\vartheta_{K,0}$ and $\Theta_{K,0}$ is exactly $t_K^\star$ defined in \eqref{eq:tNstar}.
\end{lemma}

\begin{remark}[Transience time for macroscopically supercritical and subcritical configurations]
\label{rem:sctransience}

In the case where $s = \delta K$ for $0<\delta<1$, one can obtain an upper-bound $C'_\varepsilon K^2\log(1/\delta)$ instead of \eqref{def:thetaK2}. This can be achieved by choosing for example $a=K/2$ in \eqref{eq:survival4}. So there exists a constant $t_\delta$ such that 
\eqs{\limsup_{K\to\infty}p_{\delta K,K}^{\SWT}(t_\delta K^2)=0.}
This estimate shows that in the presence of a macroscopic amount of excess particles, the transience time's scaling changes and becomes diffusive (of order $K^2$), rather than $\mathcal{O}(K^2\log K)$.

\medskip

For macroscopically \emph{subcritical} configurations, however, the transience time remains of order $K^2\log K$. To illustrate this point, it is enough to consider two configurations, one with a single trap of depth $K$ and the other with a single trap of depth $K+{\delta} K$, which will have exactly the same transience time.
\end{remark}

In light of \eqref{eq:critSWT}, Theorem \ref{thm:cutoff_SWT} is a direct consequence of Lemma \ref{lem:t_tran_critique} with $s=0$. 
To prove the lower bound in \eqref{eq:encadrement_theta}, it is enough to have an estimate for the case with a single trap and $s$ excess particles. For the upper bound, we also need to couple the SWT with a non-periodic SSEP in contact with empty reservoirs. The setup, however, is much more involved, because in the general case, provided the configuration is still transient at time $t$, the position of the remaining traps is random, so that we split $\T_K$ in smaller boxes that can contain the remaining trap.

\medskip
\begin{proof}[Proof of Lemma \ref{lem:t_tran_critique}] 
\textbf{Lower bound. }
For $S(\xi) = 0$, \eqref{def:thetaK1} is the lower bound in Lemma \ref{lem:unique_well}. For $S(\xi)=s > 0$, the arguments are very much the same as in the proof of the lower bound in Lemma \ref{lem:unique_well}, but we choose instead of \eqref{eq:deftau}, the reference time 
\eq{eq:tauKS}{\tau_{K,s}^{\star} := \inf\{t \ge 0 = \E^{K}_{\un}[|\sigma(t)|]\le s\}.}
The latter can be estimated through very similar calculations as in \cite[Section 5]{tran_cutoff_2022}, and we can show that 
\eq{eq:tauKS2}{\sup\limits_{s\le K} \left|\tau_{K,s}^{\star} - \frac{K^2}{\pi^2}\log\frac{K}{s}\right| = \mathcal{O}(K^2),} 
which yields \eqref{def:thetaK1} for $S(\xi) > 0$.  \\

\textbf{Upper bound. }
Let $\noteblue{Q = \lceil\log K\rceil}$ and $\noteblue{\ell=\ell_K:=\lfloor\frac{K}{Q}\rfloor}$, and split $\T_K$ into  $Q+1$ disjoint sets $A_1,\dots ,A_{Q+1}$ of size at most $\ell_K$, defined for $i\in \{1,\dots,Q\}$
\eqs{A_i=\left\{(i-1)\ell+1,i\ell\right\}:=\{k_i, \dots,k_{i+1}-1\},\qquad \mbox{ and }\qquad  A_{Q+1}=\{Q\ell+1,\dots,K\}.}
If $\xi$ is transient at time $t$, one of the $A_i$'s must still contain a trap of positive depth,  so that by union bound
\begin{equation}
\label{eq:splitAi}
\bP_\xi^K\big(\xi(t) \hbox{ is transient}\big) \le \sum_{i=1}^{Q+1} \bP_\xi^K\big(\exists k \in A_i, \xi_k(t)<0\big).
\end{equation}
Now it remains to upper bound the probability that there is a trap in a given sector $A_i$ at time $t$. This will be achieved by coupling the SWT with SSEPs with empty reservoirs, which is the purpose of Section \ref{subsec:unroll}. Then, using Corollary \ref{cor:survivalSWT}, applied to $\varepsilon /(Q+1)$, we conclude that for $t \ge \vartheta_{K,s}(\varepsilon)$, the transience probability is less than $\varepsilon$.

\end{proof}

\medskip
\subsection{Coupling with SSEPs in contact with reservoirs}
\label{subsec:unroll}

In order to complete our proof, we wish to couple a SSEP with traps $\xi$ with a non-periodic SSEP in contact with empty reservoirs, on the event 
\eq{eq:DefEAt}{E_A(t):=\big\{\exists k \in A,\; \xi_k(t)<0\big\},}
appearing in \eqref{eq:splitAi}, on which 
there remains a trap in a fixed segment $A$ at time $t$. Our main result of this section is Corollary \ref{cor:survivalSWT} below, which estimates the $\bP_\xi^K$-probability of $E_A(t)$. The presence of a trap in $A$ breaks the periodicity of the system, in the sense that no particle can have fully crossed $A$ before time $t$. In particular, we couple the  SWT on $\T_K$  with SSEP-like processes on the \emph{non-periodic} segment $\llbracket 1, K+a \rrbracket$, where $\noteblue{a:=|A|}$, with empty reservoirs at the boundaries.
The main idea behind this coupling it to ``unroll'' $\T_K$ (with $A$ on one side) and an additional copy of $A$ on the other side. We then couple particles' trajectories in $\llbracket 1, K+a \rrbracket$, modulo $K$,  with the periodic trajectories in the SWT.  The survival of a particle in the unrolled process will imply that of the corresponding one in $\xi$, and will allow us to estimate the survival probability in $\xi$.

\medskip

More precisely, we will couple $\xi$ with three unrolled processes, depending on the initial positions of the particles: outside of $A$, in the left copy of $A$ and in the right copy of $A$. This construction will help us keep track of the survival of individual particles, and each of the three resulting processes will be upper-bounded by a SSEP in contact with empty reservoirs, provided $\xi(t)$ has a trap in $A$.

\bigskip

We place ourselves in the setting of \emph{distinguishable} particles defined in Section \ref{subsub:echange}, in order for  individual particles to behave as symmetric random walks until they get trapped. Once again, forgetting the particle's labels, we obtain a SSEP with traps.
Consider therefore a labelled SSEP with traps 
\eqs{\hat \xi = (\xi, \Xi_1, ... \Xi_n, \delta_1, ... \delta_n)(t)}
on $\T_K$ with $n\geq 1$ particles defined as in Section \ref{subsub:echange}. Recall that this means there are $n$ labelled particles, the position of particle $p$ at time $t$ is given by $\Xi_p(t)$ and $\delta_p(t)$ is 0 if the particle has been trapped before time $t$, 1 otherwise.
Fix a segment $A \subset \T_K$.  Up to translation, we can assume that $A = \llbracket 1, a\rrbracket$, and we denote by $I$ the set of labels that were not initially in $A$,
\eqs{I:=\{p\in \llbracket 1,n\rrbracket,\; \Xi_p(0)\notin A\}.}
We will use repeatedly our general convention that $\gamma:=\gamma(0)$ represents the \emph{initial state} of a time process $\gamma(\cdot)$, so that for example,  $\Xi_p(0)$ will simply be denoted by  $\Xi_p$.

\paragraph{Central unrolled coupling.} 
The aim of this first coupled process is to track the particles that were initially outside of $A$. Given $\hat \xi(\cdot)$ we build a coupled process $\Upsilon_A(t):=\Upsilon_A(\hat\xi(\cdot),t)$ through its labelled particles'  trajectories, so that we introduce
\eqs{\Upsilon_A(t) = \big\{(\Xi'_p, \delta_p')(t),\; p\in I\big\}.}
Initially, $\Upsilon_A$ only contains particles that were not present in $A$ in $\xi$, so we initially set for $p \in I$
\eq{eq:unfold}{\Xi'_p = \Xi_p \in \llbracket a+1, K \rrbracket \qquad \mbox{ and }\quad \delta'_p = \delta_p=1 .}
Once again, $\Xi'_p(t)\in \llbracket 1,K+ a\rrbracket$ represents the position at time $t$ of the particle labelled $p$, whereas $\delta_p'(t)$ will be $1$ if the particle labelled $p$ is still alive in the system at time $t$, and $0$ if it has been killed before time $t$. This initial mapping is illustrated in Figure \ref{fig:sigma_omega_init}.

\begin{figure}
\centering
\includegraphics[scale=1]{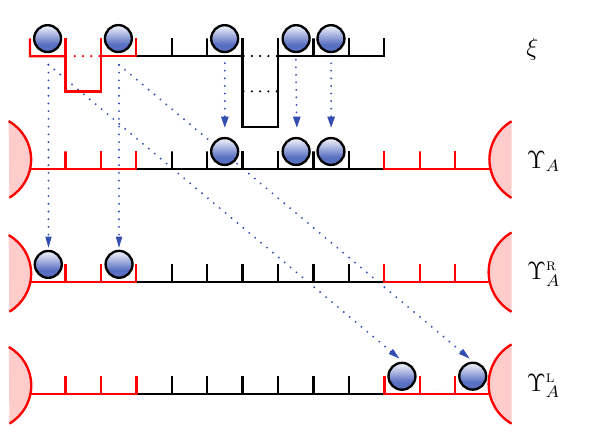}
\caption{Illustration of the initial mapping of $\hat \xi$ to $\Upsilon_A$, $\Upsilon_A^{\textsc{r}}$ and $\Upsilon_A^{\textsc{l}}$, where $A$ is in red. The difference between these three processes lies in the placement of initial particles.}
\label{fig:sigma_omega_init}
\end{figure}

As in Section \ref{subsub:echange}, we define $\hat \xi$ by equipping each edge $(k,k+1)$ with a clock $\mathscr{T}_{k}$, that can trigger jumps or swaps of the labelled particles present at the edge. To ensure that trajectories in $\Upsilon_A(\cdot)$ can be compared to that of $\hat \xi$, $\Upsilon_A(\cdot)$ is defined with the same clocks. This means that for any $k\in \llbracket 0,K+ a\rrbracket$, we associate to the edge $(k,k+1)$ the clock $\mathscr{T}_{k \mod K}$, as shown in Figure \ref{fig:sigma_omega_clock}.

\begin{figure}
\centering
\includegraphics[scale=0.85]{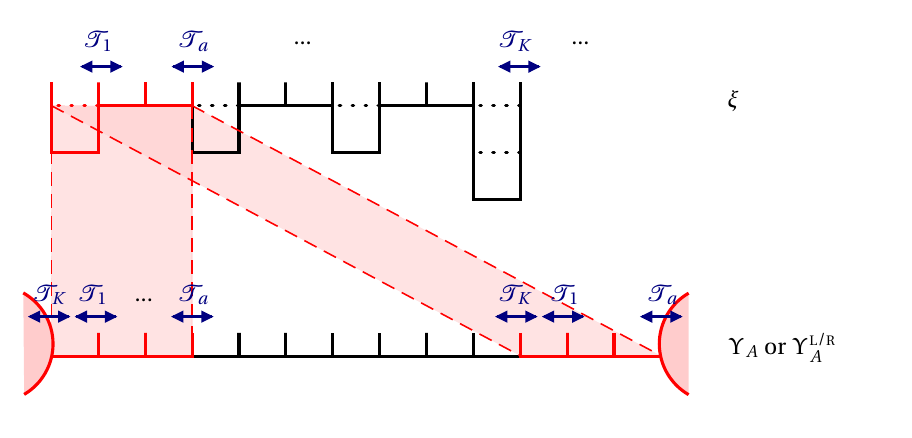}
\caption{Illustration of the mapping of the Poisson clocks from $\xi$ to $\Upsilon_A$ and $\Upsilon_A^{\textsc{l}/\textsc{r}}$, with $A$ represented in red. This ensures that the trajectories of particles in the auxiliary processes are the same as in $\hat \xi$. Notice that the Poisson clocks associated to the red zone in $\hat \xi$ are sent to both red zones in the new process.}
\label{fig:sigma_omega_clock}
\end{figure}

In this description, the boundary edges $(0,1)$ and $(K+a,K+ a+1)$ will represent reservoir interactions at the boundaries. Note that in particular, a single clock ring can affect two associated edges in $\Upsilon_A(\cdot)$ if the clock's edge intersects $A\subset \T_K$, so this process will not \emph{a priori} behave as a classical exclusion process, because two distant particles could jump at the same time. However we will see that on the event $E_A(t)$ (defined in \eqref{eq:DefEAt}), there remains a trap in $A$ so this doesn't occur, as illustrated in Figure \ref{fig:conf_xi}. 

\begin{figure}
\centering
\includegraphics[scale=0.9]{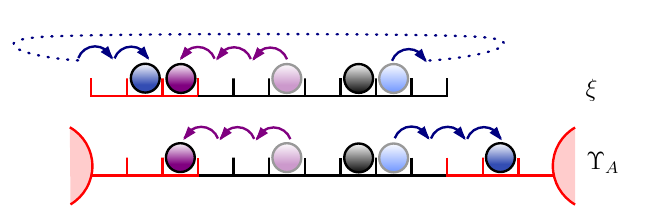}
\caption{Illustration of the case where 2 distant particles could be moved by the same Poisson process in $\Upsilon_A(\cdot)$. This can happen if a particle, here the blue one, reaches the red zone arriving from the right, another one, the purple one, reaches the red zone by arriving from the left. These particles end up being neighbours in $\hat\xi(\cdot)$. For this to happen, there must be no trap in the red zone, otherwise these particles couldn't become neighbours. }
\label{fig:conf_xi}
\end{figure}

Given these Poisson clocks, the evolution of $\Upsilon_A$ is coupled to that of $\xi$, in the following way. We only describe here $\Upsilon_A(\cdot)$'s evolution, since $\hat \xi(\cdot)$'s was described in Section \ref{subsub:echange}. Assume that for some $k\in \T_K$ a clock $\mathscr{T}_{k}$ rings at time $t$. Then, at each associated edge $e':=(k', k'+1)$ in $\llbracket 0,K+ a+1\rrbracket$, as illustrated in Figure \ref{fig:sigma_omega_cas}, we proceed in the following way:
\begin{enumerate}[i)]
\item If there is no $p\in I$ such that $\Xi'_p(t^-)\in e'$, nothing happens in $\Upsilon_A(t)$. 
\item If there is at least one live particle labelled $p\in I$ in $e'$, and $\Xi'_p(t^-)=k'$ (resp. $\Xi'_p(t^-)=k'+1$), then the particle jumps over $e'$, so that we set $\Xi'_p(t)=k'+1$ (resp. $\Xi'_p(t)=k'$). In particular, if another live particle labelled $q\in I$ was on the other site of the edge, the two particles swap positions in $\Upsilon_A(t)$. 
\item If at time $t$, a particle labelled $p\in I$ gets trapped in $\hat\xi$, then we also kill it in $\Upsilon_A(t)$ by setting $\delta_p'(t)=\delta_p(t)=0$.
\item If $k'=0$ or $k'=a+K$, then the clock is associated with one of the boundaries. Then, if a particle labelled $p\in I$ is alive ($\delta_p'=1$) at the corresponding boundary site $k'=1$ or $k'=a+K$ then it is killed as well, still by setting $\delta_p'(t)=0$.
\end{enumerate}

\begin{figure}
\centering
\includegraphics[scale=1]{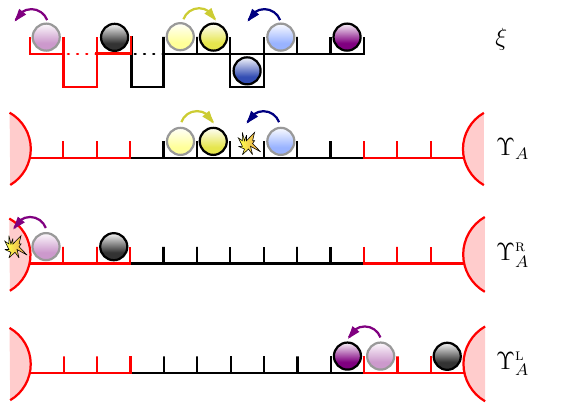}
\caption{Illustration of the joint evolution of $\xi(\cdot)$, $\Upsilon_A(\cdot)$ $\Upsilon_A^{\textsc{r}}(\cdot)$ and $\Upsilon_A^{\textsc{l}}(\cdot)$ with different sorts of jumps. The blue particle jumps to a trap in $\hat\xi(\cdot)$ so it is killed in $\Upsilon_A(\cdot)$, which is symbolised by an explosion. The yellow particle simply jumps to an empty site in $\hat\xi(\cdot)$, so the same jump occurs in $\Upsilon_A(\cdot)$. The purple particle jumps from site $1$ to $K$ in $\hat\xi(\cdot)$, so the corresponding particle in $\Upsilon_A^{\textsc{r}}(\cdot)$ jumps to the left and is killed by the reservoir, while in $\Upsilon_A^{\textsc{l}}(\cdot)$ it simply jumps to the left.}
\label{fig:sigma_omega_cas}
\end{figure}

These dynamics are summarised in Figure \ref{fig:sigma_omega_cas}. Note that, according to the last two rules, we must have at all times $\delta_p'(t)\leq \delta_p(t)$. Further note that, by construction, if a particle is alive in both $\hat \xi(t)$ and $\Upsilon_A(t)$, it has had exactly the same trajectory in  both processes (same jump times and  directions) at least up to time $t$. 

\medskip

If $0 \leq k\leq a$, then when the clock $\mathscr{T}_k$ rings, it affects two edges $(k,k+1)$ and $(k+K, k+K+1)$ in $\Upsilon_A(\cdot)$, and we denote $\Upsilon_A^{k; K+k}$ the resulting configuration after both edges have been updated. If instead  $a<k\leq K-1$, then the single edge $(k,k+1)$ in $\Upsilon_A$ was affected, and we denote by $\Upsilon_A^k$ the resulting configuration after the update. Similarly, we denote by $\hat \xi^k$ the configuration after edge $(k,k+1)\in \T_K$ has been updated. Then, with these notations, the pair $(\hat \xi, \Upsilon_A)(t)$ is a Markov process with joint generator 
\eqs{\mathscr{L}f(\hat \xi, \Upsilon) = \sum_{k=0}^a \left\{f(\hat \xi^{k}, \Upsilon^{k; K+k})-f(\hat \xi, \Upsilon) \right\} +  \sum_{k=a+1}^{K-1} \left\{f(\hat \xi^{k}, \Upsilon^{k})-f(\hat \xi, \Upsilon)\right\}.}
We claim the following.
\begin{lemma} \label{lem:survive_central}
Consider a labelled SSEP with traps $\hat \xi(\cdot)$ and a segment $A$. On the event $E_A(t)$ defined in \eqref{eq:DefEAt}, a labelled particle initially in $\T_K \setminus A$ survives in $\hat \xi(t)$ iff the corresponding particle survives in $\Upsilon_A(t)$.
\end{lemma}

\begin{proof}[Proof of Lemma \ref{lem:survive_central}]
First notice that by construction, for any $p\in I$, $\delta'_p(t)\le \delta_p(t)$, so that if a particle survives in $\Upsilon_A(t)$, it also survives in $\hat \xi(t)$. Assume  now that there exists a particle $p\in I$ such that $\delta'_p(t)=0$ and $\delta_p(t) = 1$, we prove that $E_A(t)$ cannot hold. 

\medskip

The only discrepancies between $\delta$ and $\delta'$ occur when a particle is killed by a reservoir in $\Upsilon_A(\cdot)$ and not in $\hat\xi(\cdot)$ therefore particle $p$ reached a reservoir before time $t$. By assumption,  initially, $\Xi'_p \in \llbracket a+1, K\rrbracket$ and the reservoirs are at sites $0$ and $K + a + 1$, this means particle $p$ fully crossed either the box $\llbracket 0,a\rrbracket $ (leftwards) or the box $\llbracket K,K+a+1\rrbracket $ (rightwards) in $\Upsilon_A(\hat\xi)$ before time $t$. Recall that trajectories $\Xi_p(\cdot)$ and $\Xi'_p(\cdot)$ are identical as long as $\delta_p(s)=\delta_p'(s)=1$. Therefore, before time $t$, in $\hat\xi$, particle $p$ was able to fully cross $A$ without dying, which proves $E_A(t)$ cannot hold.
\end{proof}
\begin{lemma} \label{lem:pas_distant_central}
On the event $E_A(t)$, for all $s \le t$, for any particles labelled $p$, $q \in I$ such that $\delta'_p(s)=\delta'_q(s)=1$, $|\Xi'_p(s) - \Xi'_q(s)| < K-1$. In particular, on $E_A(t)$, for $0 \le k \le a$ and $s \le t$, a clock ring of $\mathscr{T}_k$ never affects two distant particles in $\Upsilon(s)$.
\end{lemma}

\begin{proof}[Proof of Lemma \ref{lem:pas_distant_central}]
Let $t\ge 0$ and assume that $E_A(t)$ holds. Then, there exists $k \in \llbracket 1,a\rrbracket$ such that $\xi_k(t)<0$. So, for all $s \le t$, for all particle labelled $p$ alive at time $s$, $k < \Xi'_p(s) < K + k $ (otherwise, particle $p$ would have crossed the trap at site $k$ before time $s$ and therefore could not be alive). This is true for any other particle labelled $q$ alive at time $s$, so $|\Xi'_p(s) - \Xi'_q(s)| < K-1$. 

Let $k \in \llbracket 0, a\rrbracket$ and $s\le t$. For clock $\mathscr{T}_k$ to affect two distant particles in $\Upsilon(s)$, there should be two live particles $p$ and $q$ such that $\Xi'_p(s) \in \{k,k+1\}$ and $\Xi'_q(s) \in \{K + k,K + k+1\}$, but then we would have $|\Xi'_p(s) - \Xi'_q(s)| \ge K-1$, which contradicts the previous result.
\end{proof}

We now define the other couplings.

\paragraph{Right unrolled coupling (with suppression in $A$).} 
The aim of this second coupled process is to track the particles that were initially in A, to the right of a remaining trap. We therefore define
\eqs{I':=\llbracket 1,n\rrbracket \setminus I=\{p\in \llbracket 1,n\rrbracket,\; \Xi_p(0)\in A\},}
namely the set of labels of particles that were initially in $A$. Given $\hat \xi$ we now build a second coupled process $\Upsilon_A^{\textsc{r}}(\cdot)$ through its particle's labelled trajectories, so that we define 
\eqs{\Upsilon_A^{\textsc{r}}(t) = \big\{(\Xi^{\textsc{r}}_p, \delta_p^{\textsc{r}})(t),\; p\in I'\big\}.}
Initially, $\Upsilon_A^{\textsc{r}}$ only contains the particles that were present in $A$ in $\xi$, so that we initially set for $p \in I'$
\eq{eq:unfold2}{\Xi^{\textsc{r}}_p = \Xi_p \in \llbracket 1, a \rrbracket  \qquad \mbox{and }\quad \delta^{\textsc{r}}_p = \delta_p=1.}
The initial mapping is illustrated in Figure \ref{fig:sigma_omega_init}. 
Once again, $\Xi^{\textsc{r}}_p(t)\in \llbracket 1,K+ a\rrbracket$ represents the position at time $t$ of the particle labelled $p$, and $\delta_p^{\textsc{r}}(t)$ is $1$ or $0$ depending on whether the latter has been killed before time $t$.
\medskip

As for $\Upsilon_A(\cdot)$, to ensure that trajectories in $\Upsilon_A^{\textsc{r}}(\cdot)$ can be compared to that of $\hat \xi$, for any $k\in \llbracket 0,K+ a\rrbracket$, we associate to the edge $(k,k+1)$ the clock $\mathscr{T}_{k \mod K}$, as shown in Figure \ref{fig:sigma_omega_clock}. 
\begin{figure}
\centering
\includegraphics[scale=0.9]{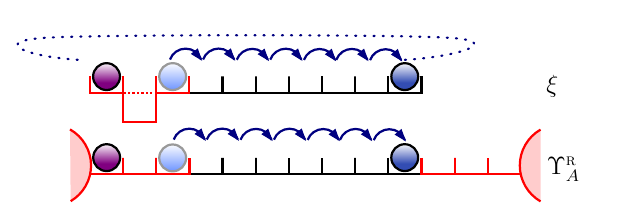}
\caption{Illustration of the case where 2 particles become neighbours in $\xi(\cdot)$ while being distant in $\Upsilon_A^{\textsc{r}}(t)$. When this happens, the particle in the left red zone (here, the purple one) is killed in $\Upsilon_A^{\textsc{r}}(t)$ by applying the map $s_{\textsc{l}}$, and if there remains a trap in $A$ then this particle was necessarily to its left.}
\label{fig:conf_om}
\end{figure}
The evolution of $\Upsilon_A^{\textsc{r}}(\cdot)$ follows the same rules i)-iv) above as $\Upsilon_A(\cdot)$, except that the set of affected labels $I$ is replaced by $I'$. However, on top of those rules, we need to make sure that no two distant live particles can be affected by the same Poisson ring, which,  as illustrated in Figure \ref{fig:conf_om}, can happen even if there is still a trap in $A$. In that case, the two particles are neighbours in $\hat \xi(t)$, although they are distant in $\Upsilon_A^{\textsc{r}}(t)$. To prevent this issue, we add the following rule to i)-iv) to complete the description of the evolution of $\Upsilon_A^{\textsc{r}}(\cdot)$:
\begin{enumerate}
\item [v)] [\emph{Suppression in $A$}] If after a jump at time $t$ has been performed, there exist two live particles labelled $p, q\in I'$ such that $\Xi_q^{\textsc{r}}(t) = \Xi_p^{\textsc{r}}(t) + K-1$, then we kill the particle labelled $p$ by setting $\delta_p^{\textsc{r}}(t)=0$. Note that this can only happen if $\Xi_p^{\textsc{r}}(t) \in A=\llbracket 1,a \rrbracket$.
\end{enumerate}
Adding this rule ensures that there will be no simultaneous jumps of far-away live particles in $\Upsilon_A^{\textsc{r}}(\cdot)$.
Notice that, if $\xi(t)$ has a trap in $A$, a particle that was killed by rule v) before $t$ is necessarily to its left. Indeed, there can be no trap between site $1$ and this particle in $\xi(t)$, since another particle was able to cross this zone in $\xi$ to become its left neighbour. 
Furthermore, we still have at all times $\delta_p^{\textsc{r}}(t)\leq \delta_p(t)$, and by construction, if a particle is alive in both $\hat \xi(t)$ and $\Upsilon_A^{\textsc{r}}(t)$, it has had exactly the same trajectory in  both processes (same jump times and  directions) at least up to time $t$.

\medskip

In order to write down $\Upsilon_A^{\textsc{r}}(\cdot)$'s Markov generator, define a ``suppression in $A$'' map $s_{\textsc{r}}$ as follows:
\eqs{
s_{\textsc{r}}:(\Xi_{p}^{\textsc{r}},\delta_{p}^{\textsc{r}})_{p\in I'}\; \longmapsto\; (\Xi_{p}^{\textsc{r}}, \tilde \delta_{p}^{\textsc{r}})_{p\in I'}}
where for $p\in I'$
\eqs{\tilde \delta_p^{\textsc{r}}  = \delta_p^{\textsc{r}} \left(1 - \un_{\{\Xi_p^{\textsc{r}} \le a+1, \exists q \in I':\; \Xi_q^{\textsc{r}} = \Xi_p^{\textsc{r}} + K-1, \delta^{\textsc{r}}_q = 1\}} \right).}
As wanted, this map kills any live particle with label $p$ in $\llbracket 1, a+1\rrbracket$, provided there is a live particle $q$ at position  $\Xi^{\textsc{r}}_p + K - 1$. 
Using the same notations as before, the pair $(\hat \xi, \Upsilon_A^{\textsc{r}})(t)$ is a Markov process with joint generator 
\eqs{\mathscr{L}_{\textsc{r}}f(\hat \xi, \Upsilon) = \sum_{k=0}^a \left\{f(\hat \xi^{k}, s_{\textsc{r}}(\Upsilon^{k; K+k}))-f(\hat \xi, \Upsilon) \right\} +  \sum_{k=a+1}^{K-1} \left\{f(\hat \xi^{k}, s_{\textsc{r}}(\Upsilon^{k}))-f(\hat \xi, \Upsilon)\right\}.}
\begin{lemma}\label{lem:right_k}
For $k \in A$, we have for any $p\in I'$, $t>0$
\eq{eq:eventright}{F_{k,p}^{\textsc{r}}(t):=\big\{k < \Xi_p\leq a,\; \xi_k(t) < 0,\;\delta_p(t)=1 \big\}\quad \Longrightarrow \quad  \delta_p^{\textsc{r}}(t)=1.}
In other words, if particle $p$ was initially to the right of $k$, and at time $t$ it is still alive in $\hat\xi$, and there is a trap at site $k$, then particle $p$ is also alive in $\Upsilon_A^{\textsc{r}}(t)$. 

\medskip

In particular, defining  $\noteblue{k_A^-(t) = \inf \{k \in A,\; \xi_k(t) < 0\}}$  the leftmost trap in $A$ at time $t$,
\eq{eq:eventkAt}{E_A(t)\cap F_{k_A^-(t),p}^{\textsc{r}}(t)\quad \Longrightarrow \quad  \delta_p^{\textsc{r}}(t)=1.}
\end{lemma}

\begin{proof}[Proof of Propositions \ref{lem:right_k}]
Assume that $\delta_p(t)=1$, if $\delta_p^{\textsc{r}}(t)=0$ only two things can have happened: 
\begin{itemize}
\item Either particle $p$ was killed by a reservoir at time $s\leq t$ in $\Upsilon_A^{\textsc{r}}(\cdot)$, but this is not possible on $F_{k,p}^{\textsc{r}}(t)$ because particle $p$ has the same trajectory up to time $s$ in both $\hat\xi(\cdot)$ and $\Upsilon_A^{\textsc{r}}(\cdot)$, so that due to the trap at site $k$, we must have $\forall s'<s$
\eqs{k<\Xi^{\textsc{r}}_p(s')< K+k.}
\item Or at time $s\leq t$,  particle $p$ was suppressed in $A$ by another live particle labelled $q\in I'$, but  this is not possible either on $F_{k,p}^{\textsc{r}}(t)$. Indeed, on $F_{k,p}^{\textsc{r}}(t)$,  $\Xi^{\textsc{r}}_q$ must also have been to the right of $k$ initially (otherwise it would have been killed by the reservoir or the trap before reaching site $K$), so that by the same argument as in the previous case, for any $s'<s$, 
\eqs{k<\Xi^{\textsc{r}}_p(s')\qquad \mbox{ and }\qquad \Xi^{\textsc{r}}_q(s')<K+k,}
and we cannot have $\Xi_q^{\textsc{r}}(s') = \Xi_p^{\textsc{r}}(s') + K-1$.
\end{itemize}
 The second implication \eqref{eq:eventkAt} is a direct consequence of the first.
\end{proof}

We now perform a similar coupling but with particles initially in the right copy of $A$.
\paragraph{Left unrolled coupling (with suppression in $\llbracket K+1,K+a \rrbracket$).} 
We finally build a third coupled process $\Upsilon_A^{\textsc{l}}(\cdot)$ which will be very similar to $\Upsilon_A^{\textsc{r}}(\cdot)$. Its aim is to track the particles initially in $A$, but to the left of a trap. It is defined with all particles starting in the right copy of $A$, namely $\noteblue{\tilde{A}:=\llbracket K+ 1, K + a \rrbracket}$ and an analogous ``suppression in $\tilde{A}$'' mechanism. More precisely, define
\eqs{\Upsilon_A^{\textsc{l}}(t) = \big\{(\Xi^{\textsc{l}}_p, \delta_p^{\textsc{l}})(t),\; p\in I'\big\}.}
Initially, $\Upsilon_A^{\textsc{l}}$ only contains in $\tilde{A}$ the particles that were present in $A$ in $\xi$, so that we set for $p\in I'$
\eq{eq:unfold3}{\Xi^{\textsc{l}}_p = K + \Xi_p \in \tilde{A}  \qquad \mbox{and }\quad \delta^{\textsc{l}}_p = \delta_p=1.}
The initial mapping is illustrated in Figure \ref{fig:sigma_omega_init}. 
We adopt similar notations as in the previous paragraph, and $\Upsilon_A^{\textsc{l}}(\cdot)$ follows the same rules as $\Upsilon_A^{\textsc{r}}(\cdot)$, with the exception that rule v) is replaced by 
\begin{enumerate}
\item [v')] [\emph{Suppression in $\tilde{A}$}] If after a jump at time $t$ has been performed, there exists two live particles labelled $p, q\in I'$ such that $\Xi_p^{\textsc{l}}(t) = \Xi_q^{\textsc{l}}(t) + K-1$, then we kill the particle labelled $p$ by setting $\delta_p^{\textsc{l}}(t)=0$.
\end{enumerate}
 Again, this situation could happen even if there is a trap in $A$, as illustrated by Figure \ref{fig:conf_om}. This time, if $\xi(t)$ has a trap in $A$, a particle that was killed by rule v') before $t$ is necessarily to its right. Indeed, there can be no trap between this particle and site $K+a$ in $\xi(t)$, since another particle was able to cross this zone in $\xi$ to become its right neighbour. We define accordingly a ``suppression  in $\tilde{A}$'' map $s_{\textsc{l}}$
 
 \eqs{
s_{\textsc{l}}:(\Xi_{p}^{\textsc{l}},\delta_{p}^{\textsc{l}})_{p\in I'}\; \longmapsto \; (\Xi_{p}^{\textsc{l}}, \tilde \delta_{p}^{\textsc{l}})_{p\in I'}}
where for $p\in I'$
\eqs{\tilde \delta_p^{\textsc{l}}  = \delta_p^{\textsc{l}} \left(1 - \un_{\{\Xi_p^{\textsc{l}} \ge K, \exists q \in I':\; \Xi_p^{\textsc{l}} = \Xi_q^{\textsc{l}} + K-1, \delta^{\textsc{l}}_q = 1\}} \right).}
There again, at any time $t$, $\delta_p^{\textsc{l}}(t)\leq \delta_p(t)$, and by construction, if a particle is alive in both $\hat \xi(t)$ and $\Upsilon_A^{\textsc{l}}(t)$, it has had exactly the same trajectory in  both processes (same jump times and  directions) at least up to time $t$. 
As for $\Upsilon_A^{\textsc{r}}(t)$, and using the same notations as before, the pair $(\hat \xi, \Upsilon_A^{\textsc{l}})(t)$ is a Markov process with joint generator 
\eqs{\mathscr{L}_{\textsc{l}}f(\hat \xi, \Upsilon) = \sum_{k=0}^a \left\{f(\hat \xi^{k}, s_{\textsc{l}}(\Upsilon^{k; K+k}))-f(\hat \xi, \Upsilon) \right\} +  \sum_{k=a+1}^{K-1} \left\{f(\hat \xi^{k}, s_{\textsc{l}}(\Upsilon^{k}))-f(\hat \xi, \Upsilon)\right\}.}
For this left coupling, we can state an analogous result as for the right coupling, namely:
\begin{lemma}
\label{lem:left_k}
For $k \in A$, we have for any $p>m$, $t>0$
\eq{eq:eventleft2}{F_{k,p}^{\textsc{l}}(t):=\big\{1 \leq  \Xi_p < k,\; \xi_k(t) < 0,\;\delta_p(t)=1 \big\}\quad \Longrightarrow \quad  \delta_p^{\textsc{l}}(t)=1.}
In other words, if particle $p$ was initially to the left of $k$ in $\hat\xi$, and at time $t$ it is still alive in $\hat\xi$, and there is a trap at site $k$, then particle $p$ is also alive in $\Upsilon_A^{\textsc{l}}(t)$. 

\medskip

In particular, defining  $\noteblue{k_A^+(t) = \sup \{k \in A,\; \xi_k(t) < 0\}}$,
\eq{eq:eventkAt2}{E_A(t)\cap F_{k_A^+(t),p}^{\textsc{l}}(t)\quad \Longrightarrow \quad  \delta_p^{\textsc{l}}(t)=1.}
\end{lemma}

By symmetry, Lemma \ref{lem:left_k} is proved in the exact same way as Lemma \ref{lem:right_k}, we do not repeat the proof. Recall that the indicator $\delta'_p(t)$ was defined for $p\in I$, and that both $\delta_p^{\textsc{r}}(t)$ and $\delta_p^{\textsc{l}}(t)$ were defined for $p\in I'$. We extend for convenience both these definitions to the whole segment $\llbracket 1,n\rrbracket$, by setting $\forall t\geq 0$
\eqs{\delta_p'(t)=0\quad \mbox{ for }\; p\in I'\qquad\mbox{ and }\qquad \delta_p^{\textsc{r}}(t)=\delta_p^{\textsc{l}}(t)=0 \quad \mbox{ for }p\in I.}
We can now estimate the survival chance in $\hat\xi$ w.r.t. the one in the three auxilliary processes:
\begin{corollary}
\label{cor:survival}
For any $p\leq n$, and any segment $A\subset \T_K$
\eq{eq:survival}{E_A(t)\qquad \Longrightarrow\qquad \delta_p(t)\leq \delta'_p(t)+\delta_p^{\textsc{r}}(t)+\delta_p^{\textsc{l}}(t).}
In other words, provided there is still a trap in $A$ in $\xi(t)$, if particle $p$ is still active in $\hat\xi(t)$, it is also active in at least one of the three auxiliary processes $(\Upsilon_A,\Upsilon_A^{\textsc{r}},\Upsilon_A^{\textsc{l}})(t)$.
\end{corollary}
\begin{proof}
This result is a straightforward consequence of Lemmas \ref{lem:survive_central}, \ref{lem:right_k} and \ref{lem:left_k}. Assume that $E_A(t)$ holds and $\delta_p(t)=1$, otherwise there is nothing to show. If initially  $\Xi_p\in \T_K\setminus A$, then according to Lemma \ref{lem:survive_central} we have $\delta_p'(t)=1$. Otherwise, $\Xi_p\in A$, and we must have either  $\Xi_p\geq k_A^-(t)$, or $\Xi_p\leq k_A^+(t)$, or both of these are true. In the first case, according to Lemma \ref{lem:right_k}, $\delta_p^{\textsc{r}}(t)=1$, whereas in the second, according to Lemma \ref{lem:left_k}, $\delta_p^{\textsc{l}}(t)=1$. In the third case, using both propositions, we have $\delta_p^{\textsc{r}}(t)=\delta_p^{\textsc{l}}(t)=1$. Ultimately, in all cases one of the auxiliary processes' $\delta_p$ must be $1$, which concludes the proof.
\end{proof}

Now, only remains to estimate the survival chances in the auxiliary processes. For this purpose, given $\Upsilon_A(\cdot)$, $\Upsilon_A^{\textsc{l}}(\cdot)$, and $\Upsilon_A^{\textsc{r}}(\cdot)$,  define for $k\in \llbracket 1,K+ a\rrbracket$
\eq{eq:CPsigma}{\sigma_k(t)=\sum_{p\in I} \delta_p'(t)\un_{\{\Xi'_p(t)=k\}}, \quad \sigma^{\textsc{r}}_k(t)=\sum_{p\in I'} \delta^{\textsc{r}}_p(t)\un_{\{\Xi^{\textsc{r}}_p(t)=k\}},\quad  \sigma^{\textsc{l}}_k(t)=\sum_{p\in I'} \delta^{\textsc{l}}_p(t)\un_{\{\Xi^{\textsc{l}}_p(t)=k\}}}
the associated \emph{unlabelled} exclusion configurations $\sigma(t),\; \sigma^{\textsc{r}}(t),\; \sigma^{\textsc{l}}(t)\in \{0,1\}^{K+a}$ defined by the \emph{live particles} in each process.

\begin{proposition}
\label{prop:coupling_central}
There exists a coupling $\Theta$ between 
\begin{itemize}
\item the labelled SWT $\hat\xi(\cdot)$ (and in particular the three coupled processes $\sigma(\cdot)$, $\sigma^{\textsc{r}}(\cdot)$, $\sigma^{\textsc{l}}(\cdot)$ defined by  \eqref{eq:CPsigma}) 
\item three unlabelled SSEPs $\tilde{\sigma}(\cdot),\; \tilde{\sigma}^{\textsc{r}}(\cdot), \tilde{\sigma}^{\textsc{l}}(\cdot)$ on $\llbracket 1,K + a \rrbracket$ in contact with empty reservoirs at the boundaries. These three processes have distribution $\Q^{K+a+1}$ defined in Section \ref{sec:uniquetrap}, and respectively start from the same initial configurations \eq{eq:InitConfigcouplingTheta}{\tilde \sigma(0):=\sigma(0),\qquad \tilde \sigma^{\textsc{r}}(0):=\sigma^{\textsc{r}}(0)\qquad \mbox{ and }\qquad \tilde \sigma^{\textsc{l}}(0):=\sigma^{\textsc{l}}(0).}
\end{itemize}
This coupling is such that $\Theta$-a.s. 
\eq{eq:boundEAt}{E_A(t) \qquad \Longrightarrow \qquad \begin{cases}
\sigma(s) \le \tilde \sigma(s),\\
\sigma^{\textsc{r}}(s) \le \tilde \sigma^{\textsc{r}}(s)\\
\sigma^{\textsc{l}}(s) \le \tilde \sigma^{\textsc{l}}(s)\end{cases}
\qquad 
\forall s\leq t.}
\end{proposition}

\begin{proof}

Note that in order to define a SSEP with reservoirs on $\llbracket 1,K + a \rrbracket$, it is enough to define an initial state, and, for $k\in \llbracket 0, K+a\rrbracket$ independent, rate $1$ Poisson clocks $\widetilde{\mathscr{T}}_k$, each associated with an edge $ (k,k+1)$ in the system (including the boundary edges $(0,1)$ and $(K+a, K+a+1)$ linking to the reservoirs). Recall that $\hat\xi$ (and its auxiliary processes) was built using a set of Poisson clocks $\mathscr{T}_k$ for $k\in \T_K$, and that this set was extended to $k\in \llbracket 0, K+a\rrbracket$ by $\mathscr{T}_k:=\mathscr{T}_{k \!\mod \! K}$. Consider now a second set of rate $1$, independent clocks $\mathscr{S}_k$, this time defined on the extended set $k\in \llbracket 0, K+a\rrbracket$. Given the auxiliary processes $\sigma(\cdot)$, $\sigma^{\textsc{r}}(\cdot)$, $\sigma^{\textsc{l}}(\cdot)$, we define for $k\in \llbracket 0, K+a\rrbracket$, and for $*\in\{\;\,,\textsc{r},\textsc{l}\}$
\eq{eq:Poisson2}{\widetilde{\mathscr{T}}_k^*:= \Big\{t\in \mathscr{T}_k : E_A(t)\cap \{\sigma_k^*(t)+\sigma_{k+1}^*(t)\geq 1\}\Big\}\cup \Big\{t\in \mathscr{S}_k : E_A(t)^c\cup\{ \sigma_k^*(t)+\sigma_{k+1}^*(t)=0\}\Big\}}
In other words, while an edge $e:=(k,k+1)$ contains at least one \emph{live particle} in $\sigma^*$, and there is still a trap in $A$ in $\xi$, it is affected in $\tilde{\sigma}^*$ by $\mathscr{T}_{k}$, and by $\mathscr{S}_k$ otherwise. Let us now check that Equation \eqref{eq:Poisson2} defines  for $*\in\{\;\,,\textsc{r},\textsc{l}\}$ three sets of rate $1$, independent (for $k\in \llbracket 0, K+a\rrbracket$) Poisson clocks. Because $E_A(t)$ and the $\sigma_k^{*}(\cdot)$ are all observables at time $t$ of a Markov process, each $\widetilde{\mathscr{T}}_k^{*}$ is indeed a rate $1$ Poisson clock. We now  only need to check that they are indeed independent in $k\in \llbracket 0, K+a\rrbracket$. But to do so, it is enough to use the fact that the $(\mathscr{T}_k)_{k\in \T_K}$ and the $(\mathscr{S}_k)_{k\in \llbracket 0, K+a\rrbracket}$ are i.i.d. clocks, together with the observation that on $E_A(t)$, by construction, one cannot have both 
\eq{eq:activevoisins}{\sigma^*_k(t)+\sigma^*_{k+1}(t)\geq 1\qquad \mbox{ and }\qquad \sigma^*_{k+K}(t)+\sigma^*_{k+K+1}(t)\geq 1.}
The last statement is true because of Lemma \ref{lem:pas_distant_central} and because of the suppression mechanisms in $A$ and $\tilde{A}$, since \eqref{eq:activevoisins} means in particular that two live particles in $\sigma^*$ have before time $t$ been at distance less than $1$ in $\hat\xi$, but at a distance greater than $K-1$ in one of the auxiliary processes.

\medskip

We can therefore define the three SSEPs in contact with reservoirs $\tilde{\sigma}(\cdot),\; \tilde{\sigma}^{\textsc{r}}(\cdot), \tilde{\sigma}^{\textsc{l}}(\cdot)$, following respectively the three sets of clocks defined by \eqref{eq:Poisson2}, and started from the initial configurations \eqref{eq:InitConfigcouplingTheta}. We then denote by $\Theta$ the joint distribution of the $(\mathscr{T}_k)_{k\in \T_K}$ and the $(\mathscr{S}_k)_{k\in \llbracket 0, K+a\rrbracket}$. Note that in the three processes $\tilde{\sigma}(\cdot),\; \tilde{\sigma}^{\textsc{r}}(\cdot), \tilde{\sigma}^{\textsc{l}}(\cdot)$, particles can only get killed once they reach the reservoirs. In particular, on the event $E_A(t)$, live particles in $\sigma(\cdot)$ and $\tilde{\sigma}(\cdot)$ have exactly the same behaviour up to time $t$, so that the upper bounds \eqref{eq:boundEAt} follows straightforwardly, and the only discrepancies are due to particles that have been killed in $\sigma^*(\cdot)$ but not in $\tilde{\sigma}^*(\cdot)$.
\end{proof}

We are now able to derive the wanted control on the SWT, started from a configuration with $S(\xi)\geq 0$ excess particles.
\begin{corollary}
\label{cor:survivalSWT}
For any SWT configuration $\xi\in \Gamma_K$, for any segment $A\subset \T_K$, denoting by $a:=|A|$, we have for any $t>0$
\eq{eq:survival3}{{\bf P}_\xi^K (E_A(t))\leq  3\Q_{\un}^{K+a+1}\big(|\sigma(t)|> S(\xi)/3\big),}
where as before $\Q_{\un}^{K'} $ is the distribution of a SSEP on $\llbracket 1,K'-1\rrbracket$ in contact at the boundaries with empty reservoirs, and started from the full configuration $\un\equiv 1$.

\medskip

In particular there exists a constant $C$ such that, for $K$ large enough, for any $\varepsilon>0$, 
\eq{eq:survival4}{{\bf P}_\xi^K (E_A(t))\leq  \varepsilon, \qquad \forall t\geq \frac{(K+a)^2}{\pi^2}\log \frac{K}{S(\xi) \vee 1} + CK^2 \pa{1 + \log\frac{3}{\varepsilon} },}
\end{corollary}
\begin{proof}
We use the labelled SWT construction, and denote by $n:=|\xi^+|$ the initial number of particles. For any set $A$, if $\xi$ has $S(\xi)$ excess particles, then $E_A(t)$ implies in particular that $|\xi^+(t)|> S(\xi)$, because at least one trap has not been filled. But by construction, $|\xi^+(t)|=\sum_{p=1}^n \delta_p(t)$, so that by Corollary \ref{cor:survival}, 
\begin{align*}E_A(t)\qquad &\Longrightarrow  \qquad S(\xi)<|\xi^+(t)|\leq  |\sigma(t)|+|\sigma^{\textsc{r}}(t)|+|\sigma^{\textsc{l}}(t)|\\
 &\Longrightarrow  \qquad \max(|\sigma(t)|,|\sigma^{\textsc{r}}(t)|,|\sigma^{\textsc{l}}(t)|)>S(\xi)/3.\end{align*}
We now use the coupling of these three processes with the SSEPs with empty reservoirs laid out in Proposition \ref{prop:coupling_central}, together with a union bound, to finally obtain that
\eq{eq:survival5}{{\bf P}_\xi^K (E_A(t))\leq \sum_{*\in\{\;\,,\textsc{r},\textsc{l}\}}\Q_{\sigma^*(0)}^{K+a+1}\big(|\tilde{\sigma}^*(t)|> S(\xi)/3\big),}
where the initial configurations were defined by \eqref{eq:unfold}, \eqref{eq:unfold2}, \eqref{eq:unfold3} and \eqref{eq:CPsigma}. The SSEP's attractiveness then allows to upper bound by the full initial configuration which proves \eqref{eq:survival3}.

\medskip

Equation \eqref{eq:survival4} for $S(\xi) = 0$ (the critical case) is a direct consequence of \eqref{eq:survival3} and the upper bound in Lemma \ref{lem:unique_well}, applied to a SSEP with empty reservoirs on $\llbracket 1, K+a \rrbracket$. 

\medskip

For $S(\xi) > 0$, the arguments are the same as in the proof of the upper bound in Lemma \ref{lem:unique_well}, but we choose as a reference time $\tau_{K,s}^{\star}$, defined in \eqref{eq:tauKS}, instead of $\tau_{K}^{\star}$. Then the estimate \eqref{eq:tauKS2} yields \eqref{eq:survival4} for $S(\xi)=s > 0$, which completes the proof. 
\end{proof}

\section{Proof of Theorem \ref{thm:freezetime_cutoff_crit}}
\label{sec:proofthm2.2}
\subsection{Crude upper bound for the FEP's transience time}

We now wish to convert our transience time cutoff result on the SWT to the FEP. The critical case for the FEP, meaning estimating $\theta^{\FEP}_{\star,N}(\varepsilon)$ (defined analogously to the one for the SWT) can be done straightforwardly using the critical case of Theorem \ref{thm:cutoff_SWT}. However, in Theorem \ref{thm:freezetime_cutoff_crit} we are dealing with the maximal transience time $\theta^{\FEP}_{N}(\varepsilon)$, rather than the critical one. The problem is therefore that we are looking at a fixed \emph{ring size} $N$ estimate for the FEP, while our results on the SWT concern a fixed size of the ring for the SSEP with traps $K$, which corresponds to a \emph{fixed number of particles for the FEP}, rather than a fixed ring size. 

\medskip

More precisely, given a FEP configuration $\eta$ on $\T_N$, the mapped SWT has at most $N$ particles, so that 
 \begin{equation} \label{eq:encadrement_pas_sharp0}
 \theta^{\FEP}_N(\varepsilon) \le \max_{K\in \llbracket 1,N \rrbracket} \theta^{\SWT}_K(\varepsilon)=\max_{K\in \llbracket 1, N \rrbracket}\theta^{\SWT}_{\star,K}(\varepsilon)\leq \Theta_{N}(\varepsilon),
\end{equation}
where the last upper bound was defined in Lemma \ref{lem:t_tran_critique}.
\medskip

On the other hand, the $\varepsilon$-transience time needs in particular to encompass critical configurations, for which $K=N/2$, so that for even $N$ (otherwise there is no critical state)
\begin{equation} \label{eq:encadrement_pas_sharp00}
\theta^{\FEP}_N(\varepsilon) \ge \theta^{\FEP}_{\star,N}(\varepsilon)=\theta^{\SWT}_{\star,N/2}(\varepsilon).
\end{equation}
For odd $N$, the bound must hold for any configuration with $\lfloor N/2\rfloor$ particles, so that for any $N$, even or odd, we have
\begin{equation} \label{eq:encadrement_pas_sharp000}
\theta^{\FEP}_N(\varepsilon) \ge \theta^{\SWT}_{\star,\lfloor N/2\rfloor}(\varepsilon)\geq \vartheta_{\lfloor N/2\rfloor}(\varepsilon)\end{equation}
once again with the same notations as in Lemma \ref{lem:t_tran_critique}. Given Lemma \ref{lem:t_tran_critique}, \eqref{eq:encadrement_pas_sharp0} and \eqref{eq:encadrement_pas_sharp000} yield
\begin{equation} \label{eq:encadrement_pas_sharp}
t_{N/2}^\star - CN^2\pa{1+\log \frac{1}{1-\varepsilon}} \le \theta^{\FEP}_N(\varepsilon) \le t_{N}^\star + CN^2\left(1 + \log \frac{\log N}{\varepsilon}\right)
\end{equation}
for some positive constant $C$.

\medskip

Because the dominant terms $t_{N/2}^\star$, $t_N^\star$ on both sides are not the same, this bound is not enough to prove cutoff, so that we need to sharpen our upper bound to match the lower bound at the dominant order. Unfortunately, the SWT does not allow us to do so, because the mapping itself does not lend itself well to the supercritical case; in the supercritical regime, adding particles in the FEP should \emph{help} reaching the transient component faster, however adding particles in $\eta$ increases the SWT's ring size, thus making our transience estimates worse. For this reason, to get a sharp supercritical bound for the FEP's transience time, we turn to a second mapping in which adding particles decreases the ring's size. 

\subsection{Mapping to the facilitated zero-range process}

Given an integer $P\geq 1$,  we define the Facilitated Zero-Range process (FZR) on $\T_P$, which is a zero-range process with rate function $\noteblue{g(k):=\un_{\{k\geq 2\}}}$, meaning a Markov process on $\noteblue{\Omega_P:=\N^{\T_P}}$, with generator given by 
\eq{eq:geneomega}
{
{\mathscr{L}}^{\FZR}_P f (\omega)=\sum_{y\in \T_P}{\bf 1}_{\{\omega_y\geq 2\}}\big\{f(\omega^{y,y-1})+f(\omega^{y,y+1})-2f(\omega)\big\},
}
where 
\eqs{
\omega^{y,y+z}_{y'}=\begin{cases}\omega_y-1 & \mbox{ if }y'=y\\
\omega_{y+z}+1 & \mbox{ if }y'=y+z\\
\omega_{y'}& \mbox{ if }k'\neq y,y+z.
\end{cases}
}
The mapping from exclusion process to zero-range processes is very classical, and has been used repeatedly on the FEP in particular, to derive its macroscopic properties   \cite{blondel_hydrodynamic_2020,blondel_stefan_2021,EZ23}. For this reason, we simply sketch it, and refer the reader to \cite[Section 3]{EZ23} for details. To define the mapping, consider a FEP $\eta(\cdot)$, define $P=N-K$ its number of empty sites, and arbitrarily enumerate their successive positions 
\eqs{X_1(0)<X_2(0)<\dots<X_P(0).}
We follow the trajectories $X_1(t)<X_2(t)<\dots<X_P(t)$ of the empty sites as the FEP runs its course, and define 
\eq{eq:omegamapping}{\omega_y(t)=X_{y+1}(t)-X_y(t)-1\in\N}
as the number of particles between the $y$-th and $(y+1)$-th empty sites in $\eta(t)$. Similarly to the SWT mapping, we define $\Phi(\eta)=\omega$ (resp. $\Phi^\star[\eta])=\omega(\cdot)$) the statical mapping of the initial configuration (resp. the dynamical mapping of the trajectory). It is then straightforward to see (cf. \cite[Section 3.2]{EZ23}) that if $\eta(\cdot)$ is a FEP, then $\Phi^\star[\eta]=\omega(\cdot)$ defined through \eqref{eq:omegamapping} is a facilitated zero-range process started from $\noteblue{\Phi(\eta):=\Phi^\star[\eta](0)}$ and driven by the generator  \eqref{eq:geneomega}.

\medskip

Since the jump rate $g(\cdot)$ for this zero-range process is non-decreasing, it is also attractive under a similar basic coupling as the one introduced for the SWT in Section \ref{subsec:att_swt}. Note that the FZR is degenerate, in the sense that isolated particles cannot jump. For this reason, the FZR also exhibits frozen ($\omega\leq 1$), ergodic ($\omega \geq 1$) and transient (non-ergodic and non-frozen) configurations. We denote by $\noteblue{{\bf P}_\omega^P}$ the distribution of a FZR started from $\omega\in \Omega_P$ and driven by the generator ${\mathscr{L}}^{\FZR}_P$.

\subsection{Transience time of the supercritical FEP} 
\label{subsec:super_crit_FEP}
With this second mapping, we are now equipped to get a sharp estimate on the supercritical transience time for the FEP. Let $\eta$ a supercritical FEP configuration on $\T_N$. Then $\eta$ has $P\leq N/2$ empty sites, and $\omega:=\Phi(\eta)$ is a FZR configuration on $\T_P$. Furthermore, by the dynamical mapping, for any $\omega\in \Omega_P$ and any $\eta$ such that $\Phi(\eta)=\omega$ we have 
\begin{equation}
   p_\omega(t):={\bf P}_\omega^P(\omega(t)\mbox{ is transient})=p_\eta(t).
\end{equation}
Since $\omega$ is supercritical, by removing the right number of particles in $\omega$, there exists a critical FZR (with exactly $P$ particles) $\omega^\star\leq \omega$ on $\T_P$. By attractiveness, using the basic coupling between the two FZR started from $\omega, \omega^\star$, we see that if $\omega^\star(t)$ is ergodic then $\omega(t)$ is too, hence 
 \eq{eq:borntrans1}{p_\eta(t)=p_\omega(t) \le p_{\omega^\star}(t).}
Note that any $\eta^\star\in \Phi^{-1}(\omega^\star)$ is a critical FEP configuration on $\T_{2P}$, with $P \le N/2$, so that $\xi^\star := \Pi(\eta^\star)$ is a critical SWT configuration on $\T_P$. This, together with \eqref{eq:borntrans1}, yields
\eq{eq:inegmappings}{p_\eta(t)\leq    p_{\omega^\star}(t)=p_{\eta^\star}(t) = p_{\xi^\star}(t) \le p_{\star, P}^{\SWT}(t),}
where the right-hand side is the maximal critical transience probability for the SWT defined in \eqref{eq:critprobSWT}. Since the upper bound from Lemma \ref{lem:t_tran_critique} is increasing with $K$, for $t \ge \Theta_{\ent{N/2}}(\varepsilon)$ 
, then, we also have  
\eqs{t \ge \Theta_P(\varepsilon),}
so that according to Lemma \ref{lem:t_tran_critique} $p_{\star, P}^{\SWT}(t) \le \varepsilon$. In particular, for such a $t$, $p_\eta(t) \leq \varepsilon$ for all supercritical FEP configurations $\eta$. 

\medskip

For subcritical configurations $\eta$ with $K\leq N/2$ particles, we can use directly the SWT mapping, and write
\eq{eq:borntrans2}{p_\eta(t)=p_\xi(t) \le p_{\xi^\star}(t) \le p_{\star, K}^{\SWT}(t).}
where $\xi:=\Pi(\eta)$, and $\xi^\star \geq\xi$ is is a critical configuration obtained by removing trap depth in $\xi$ until it becomes critical. Reproducing the same arguments as the ones following \eqref{eq:inegmappings}, since once again $ K\leq N/2$, we obtain that for any $t \ge \Theta_{N/2}(\varepsilon)$, and any subcritical configuration $\eta$, $p_\eta(t) \leq \varepsilon$. This finally yields
\begin{equation} \label{eq:encadrement_pas_sharp2}
\theta^{\FEP}_N(\varepsilon) \le \Theta_{N/2}(\varepsilon).
\end{equation}
This bound now matches at the dominant order the lower bound already obtained in \eqref{eq:encadrement_pas_sharp}, and proves Theorem \ref{thm:freezetime_cutoff_crit}.

\appendix

\section{Exit time for random walks} \label{app:proof_RW}

We present here two classical results on the exit time of compact sets for symmetric random walks. Together, Lemma \ref{lem:RW1} and Corollary \ref{cor:1} prove \eqref{eq:explorationtime}. We follow the proof suggested by \cite[Chapter 5]{varadhan_probability_2001}. 

\begin{lemma} \label{lem:RW1}
Let $K \in \N^*$  and  $(S_i)_{i\ge 1}$ be a discrete time, simple random walk on $\Z$ starting from $x$ with $|x|<K$. Define $\tau_K = \inf\{i \ge 0 : |S_i| \ge K \}$. Then :
$\forall n \in \N$, $\forall 0 <\lambda < \frac{\pi}{2K}$,
\eqs{\PP_x(\tau_K>n) \le \frac{\cos(\lambda)^n}{\cos\lambda K}.}
\end{lemma}

\begin{proof}

Let $0< \lambda < \frac{\pi}{2K}$ : $\E_x[\cos(\lambda S_1)] = \frac12 (\cos(\lambda(x+1)) + \cos(\lambda(x-1))) = \cos \lambda \cos (\lambda x)$.

Then $Z_n = \frac{1}{(\cos \lambda)^n} \cos(\lambda S_n)$ is a martingale. Note that for all $x \in [-K,K]$, we have that $\cos(\lambda x) \ge \cos (\lambda K)$. The stopping theorem between $0$ and $\tau_K \wedge n$ gives:
\eqs{\cos (\lambda x) = \E_x\left[\frac{1}{(\cos \lambda)^{\tau_K\wedge n}} \cos(\lambda S_{\tau_K \wedge n})\right] \ge \cos (\lambda K ) \E_x \left[\frac{1}{(\cos \lambda)^{\tau_K \wedge n}}\right] .}

Since $\cos (\lambda K) > 0$, $ \E_x \left[\frac{1}{(\cos \lambda)^{\tau_K \wedge n}}\right] \le \frac{\cos(\lambda x)}{\cos (\lambda K)}$.

By monotone convergence theorem, $\E_x\left[\frac{1}{(\cos \lambda)^{\tau_K}}\right]\le \frac{\cos(\lambda x)}{\cos (\lambda N)}$.

So by Markov inequality, \eqs{\PP_x(\tau_K > n) \le \frac{\cos (\lambda x)}{\cos (\lambda K)} \cos(\lambda)^n N.}
\end{proof}

\begin{corollary} \label{cor:1}
Let $(Y_t)_{t\ge 0}$ a continuous time simple random walk on $\Z$ starting from $x$ with $|x|<K$: its jumps occur at rate 1, and independently the jump is $\pm1$ with probability $\frac12$. For $K \ge 1$, let $T_K = \inf\{t \ge 0 : |Y_t| \ge K \}$. Then there exists $C >0$ such that for all $K$, for all $s \ge 0$,
\begin{equation}
\PP_x(T_K > s) \le C e^{-\frac{s}{K^2}}.
\end{equation}
\end{corollary}
\begin{proof}
Let $s\ge 0$, $(\mathcal{N}(t))_{t\ge 0}$ a Poisson process of intensity 1, and $(S_n)_{n \ge 0}$ a simple symmetric random walk, starting from $x$, independent of $\mathcal{N}$, so that $Y_t = S_{N(t)}$ for all $t$. Let $\lambda < \frac{\pi}{2K}$. 
Then,
\begin{align}
\PP_x(T_K > s) &= \sum_{n \ge 0} \PP(\mathcal{N}(s) = n) \PP_x(\tau_K > n) \nonumber \\
&\le \sum_{n\ge 0} e^{-s}\frac{s^n}{n!} \frac{(\cos \lambda)^n}{\cos(\lambda K)} \nonumber \\
&= \frac{1}{\cos(\lambda K)} e^{-s(1-\cos \lambda)}.
\end{align}

If $\lambda \le 1$ we have $\cos \lambda \le 1-\frac{\lambda^2}{2} + \frac{\lambda^4}{24} \le 1 - \frac{11}{24}\lambda^2$.

By taking $\lambda = \sqrt{\frac{24}{11}} \frac{1}{K}$ and setting $C = \frac{1}{\cos \sqrt{24/11}}$, we then obtain 
\eqs{\PP_x(T_K > s) \le C e^{-\frac{s}{K^2}} .}
\end{proof}

\section{Upper bound for the mixing time of the SSEP}
\label{app:tmix}
Recall that $\pi_{K,s} \sim \mathcal{U}(\{\sigma \in \{0,1\}^{\T_K} : |\sigma| = s\})$, and that $\tau^{\textsc{ssep}}_{K,s}(\varepsilon)$ is the $\varepsilon$-mixing time of the SSEP on $\T_K$ with $s$ particles. We will show that
\begin{proposition}[Upper bound on the mixing time of the SSEP]
\eqs{\tau_{K,s}^{\textsc{ssep}}(\varepsilon) \le \frac{K^2}{2\pi^2} \left(\log \frac{s}{\varepsilon} + \log 4 /\pi +  o(1) \right).}
\end{proposition}
Consider a SSEP on $\T_K$ with $s$ particles, with initial configuration $\sigma$. 
Let $\zeta$ a uniform SSEP configuration, so that $\zeta \sim \pi_{K,s}$ (the invariant law of the SSEP on $\T_K$ with $s$ particles). Our aim is to couple $\zeta(\cdot)$ and $\sigma(\cdot)$, so that we label the particles of $\sigma$ and of $\zeta$ from 1 to $s$, respectively denoting by $X^1, ... , X^s$ and $Y^1, ... , Y^s$ the positions of particles labelled $1$ to $s$ in $\sigma$ and $\zeta$.

 \smallskip

We  equip the edges of $\T_K$ with rate 1 Poisson clocks $(\mathscr{T}_k)_{k \in \T_K}$, and make the labelled particles of $\sigma$ on each side of an edge swap positions when this edge rings. This labelled SSEP is in fact an interchange process, and removing the labels, we obtain a SSEP $\sigma(t)$. Notice also that the particle trajectories $(X^i_t)_{t\ge 0}$ are simple random walks.

 \smallskip

We now equip $\zeta$ with the following dynamics: consider a second, independent set of Poisson clocks $(\mathscr{S}_k)_{k \in \T_k}$. Consider an edge $(k,k+1)$: if both particle $i$ of $\zeta(t)$ and particle $i$ of $\sigma(t)$ are present on the same side of the edge (ie $Y_t^i = X_t^i \in \{k,k+1\}$), particles of $\zeta(t)$ jump across $(k,k+1)$ when $\mathscr{T}_k$ rings in order to keep the two particles coupled. Otherwise, particles of $\zeta(t)$ crosses $(k,k+1)$ when $\mathscr{S}_k$ rings.

It is straightforward to check that $(Y_t^1,..., Y_t^s)_{t \ge 0}$ is an interchange process and $\zeta(\cdot)$ is a SSEP started from $\zeta$, so that at all times $t$, $\zeta(t) \sim \pi_{K,s}$. Under this coupling, once particle $i$ of $\zeta(\cdot)$ meets particle  $i$ of $\sigma(\cdot)$, meaning after the first $t$ such that $X_t^i = Y_t^i$, the particles' trajectories are identical: $\forall t' \ge t, X_{t'}^i = Y_{t'}^i$. We denote by \noteblue{$\mathbb{Q}_{\sigma,\zeta}$} the law of this coupled labelled process. Then,

\eq{eq:dtv_ssep}{d_{TV}(\bP_\sigma^K(\sigma(t) \in \cdot), \pi_{K,s}) \le \mathbb{Q}_{\sigma,\zeta}(\exists i \le s, X_t^i \neq Y_t^i).}

To study the right hand side, we show this estimate on meeting times of random walks.
\begin{lemma} \label{lem:tau_rw}
Let $(X_t)_{t \ge 0}$, $(Y_t)_{t \ge 0}$ two independent random walks on $\T_K$, that jump symmetrically at rate 1. Let $\tau = \inf \{t \ge 0: X_t = Y_t\}$. Then for all $\varepsilon > 0$,
\eq{eq:lem_irw}{\forall t \ge \frac{K^2}{2 \pi^2} \left(\log \frac{1}{\varepsilon} + \log\frac{4}{\pi} + o(1)\right), \quad \PP(\tau > t) \le \varepsilon.}
\end{lemma}

\begin{proof}
Denote by $|X_t - Y_t|$ the length of the clockwise interval from $X_t$ to $Y_t$, then $|X_t - Y_t|$ is a random walk in $\llbracket 0, K\rrbracket$ that jumps at rate 2 and is absorbed in $\{0, K\}$. The time when $X$ and $Y$ meet is the time when $|X_t-Y_t|$ reaches 0 or $K$, so we need to estimate at what time a rate 2 random walk is absorbed by $0$ or $K$. The following computations are for a rate $1$ random walk and we will divide time by 2 at the end, they follow the same ideas as \cite[Section 5]{tran_cutoff_2022}.

Consider a SSEP on $\llbracket 1, K-1\rrbracket$ with empty reservoirs at its boundaries $\{0, K\}$, started from $\eta = \un_{K/2}$ (assuming that $K$ is even: if $K$ is odd we can use the bound for $K+1$). Then, let 
\eqs{t^{\star}_\varepsilon = \inf \{t \ge 0 : \E_{\eta}[|\eta|(t)] \le \varepsilon\}.}
By Markov inequality, the probability for the SSEP to not be empty (or the random walk not to have reached $\{0,K\}$) is less than $\varepsilon$ at times greater than $t^{\star}_{\varepsilon}$, so all we need to do is upper-bound this time.

For $t \ge 0$ and $k \in \llbracket 1, K-1\rrbracket$, let
\eqs{u_t(k) = \E_{\eta}[\eta_k(t)].}
Then by Dynkin's formula, $u$ is solution of
\begin{align}
u_0 &= \eta \\
\frac{d u_t}{dt}(k) &= \Delta u_t(k) \quad \forall k,
\end{align}
where $\Delta u(k) = u(k-1) + u(k+1) - 2u(k)$.

Setting for all $k,l \in \llbracket 1, K-1 \rrbracket$, 
\eqs{\phi_l(k) = \sqrt{2} \sin \left(\frac{\pi l k}{K}\right),}
then $(\phi_l)_{1\le l \le K-1}$ is an orthonormal basis for the inner product
\eqs{\langle f, g \rangle = \frac{1}{K} \sum_{k=1}^{K-1} f(k)g(k),}
and for all $l$:
\eqs{\Delta \phi_l = -\lambda_l \phi_l,}
where
\eqs{\lambda_l = 2 \left( 1 - \cos  \frac{\pi l}{K} \right).}
So, if $\eta = \sum_{l=1}^{K-1} c_l \phi_l$, with $c_l = \langle \eta, \phi_l \rangle$,
\eqs{u_t = \sum_{l=1}^{K-1} c_l e^{-\lambda_l t} \phi_l.}
We also have
\eq{eq:esp}{\E_{\eta}[|\eta|(t)] = \sum_{k=1}^{K-1}u_t(k) = K \langle u_t, \un \rangle,}
so if 
$\un = \sum_{l=1}^{K-1} c'_l \phi_l$, with $c'_l = \langle \un, \phi_l \rangle$, 
\eq{eq:somme}{\frac{1}{K}\E_{\eta}[|\eta|(t)] =  \sum_{k=1}^{K-1}c_l c_l' e^{-\lambda_l t}.}

We can show that for all $l$, $c_l = \frac{\sqrt{2}}{K} \sin \frac{\pi l}{2}$, and $c_l' = \frac{\sqrt{2}}{K}\frac{(\sin\frac{\pi l }{2})^2}{\tan \frac{\pi l}{2K}}$, so that
\eqs{c_l c'_l e^{-\lambda_l t} = \frac{2}{K^2}e^{-\lambda_l t} \frac{\sin \frac{\pi l}{2}}{\tan \frac{\pi l}{2K}}.}
Since $\sin \frac{\pi l}{2} = 0$ if $l$ is even, $1$ if $l \equiv 1 [4]$ and $-1$ if $l \equiv 3 [4]$, we can rewrite the sum in \eqref{eq:somme} as
\eqs{\langle u_t, \phi_l \rangle = \sum_{j=0}^{\ent{K/4}-1} (v_{4j + 1} - v_{4j+3}) + v_{K-1} \un_{K \notin 4 \Z},}
where for all odd $l$, $v_l = \frac{2}{K^2} \frac{e^{-\lambda_l t}}{\tan \frac{\pi l}{2K}}$. Notice that for all $l$, $v_l - v_{l+2} \ge 0$, so, setting $v_l = 0$ for $l \ge K$, 
\eqs{\langle u_t, \phi_l \rangle \le \sum_{j=0}^{\ent{K/4}-1} (v_{4j + 1} - v_{4j+3} + v_{4j+3} - v_{4j+5}) + v_{K-1} \un_{K \notin 4 \Z}\le v_1, }
where the last inequality is obtained by distinguishing on $K \equiv 0 [4]$ or $K \equiv 2 [4]$.

So, 
\eqs{\frac{1}{K}\E_{\eta}[|\eta|(t)] \le \frac{2}{K^2} e^{-\lambda_1 t} \frac{1}{\tan \frac{\pi}{2K}},}
which gives that
\eqs{ t_\varepsilon^{\star} \le \frac{1}{\lambda_1}\log \left(\frac{2}{\varepsilon} \frac{1}{K\tan \frac{\pi}{2K}}\right).}

Then, using that $\frac{1}{\lambda_1} = \frac{K^2}{\pi^2} + \mathcal{O}(1)$ and $\tan \frac{\pi}{2K} = \frac{\pi}{2K} + \mathcal{O}(\frac{1}{K^3})$, we get that 
\eqs{t_{\varepsilon}^{\star} \le \frac{K^2}{\pi^2} \left(\log \frac{1}{\varepsilon} + \log\frac{4}{\pi} + o(1)\right). }
To obtain \eqref{eq:lem_irw}, we divide this time by 2.
\end{proof}

We conclude by upper-bounding \eqref{eq:dtv_ssep} thanks to the previous Lemma: 
\begin{align}
    \mathbb{Q}_{\sigma,\zeta}(\exists i \le s, X_t^i \neq Y_t^i) &\le s \sup\limits_{1\le i \le s}\mathbb{Q}_{\sigma,\zeta}(|X_{t'}^i -Y_{t'}^i| \hbox{ has not reached }\{0,K\} \hbox{ before } t) \\
    &\le s \PP(\tau > t)
\end{align}
where $\tau$ was defined in Lemma \ref{lem:tau_rw}. So taking $t \ge \frac{1}{2} t^{\star}_{\varepsilon/s}$, we have $d_{TV} (\bP^K_{\sigma} (\sigma(t) \in \cdot), \pi_{K,s}) \le \varepsilon $.

\section{Absence of negative dependence for the SWT}
\label{subsec:noND}
Fix random vector $X:=(X_1, ..., X_K) \in \{0,1\}^K$, and define for $I\subset \llbracket1,K\rrbracket$ the $\sigma$-algebra  $\mathscr{F}_I:= \sigma(X_k, k\in I)$  of events depending only on the $I$-coordinates of the vector. We also say that an event $A$ is increasing if ($X\leq X'$ and $X\in A)\Rightarrow X'\in A$. Following e.g. \cite{borcea_negative_2009}, we say that a the vector $X$ is \emph{negative dependent} if $\forall I,J \subset \llbracket 1,K \rrbracket, I\cap J = \emptyset$, and for any increasing events $A\in \mathscr{F}_I$ and $B\in \mathscr{F}_J$,
\eq{eq:nd}{ \PP(A \cap B) \le \PP(A) \PP(B).}
Note that a negative dependent vector $X$ is in particular \emph{pairwise negative correlated},  meaning that 
\eq{eq:nc}{\forall k \neq k', \qquad \PP(X_k=1, X_{k'}=1) \le \PP(X_k=1)\PP(X_{k'}=1).}
For a SSEP with or without reservoirs $\sigma(\cdot)$, started from a negative dependant initial state $\sigma$, negative dependence is preserved through time (see \cite[Lemma 12]{salez_universality_2022} and \cite{tran_cutoff_2022}), and in particular \eqref{eq:nc} holds at all times $t$  for $X_k(t) = \sigma_k(t)$. 

\medskip

We now  show that, even though the SWT bears strong resemblance with the SSEP in contact with empty reservoirs, negative dependence is \emph{not} preserved by its dynamics. For a SWT $\xi$, since particles can get trapped, several choices are available to define $X_k$. More precisely, one can consider
\begin{enumerate}
\item either $X_k(t)=\un_{\{\xi_k(t)>\xi_k(0) \}},$ which indicates if there is a particle, be it live or dead, at site $k$.
\item Or $X_k'(t) = \un_{\{\xi_k(t) = 1\}},$ which indicates if there is a live particle at site $i$. Note that in this case, if the SWT is supercritical, \emph{once in the ergodic component}, it behaves as a SSEP, so if it is in a negative dependent state at this point, this will be preserved by the dynamics. 
\end{enumerate}
We now show that in both cases, negative dependence is not preserved by the SWT, by considering deterministic initial states, so that to prove that negative dependence is not preserved, it is enough to show that there is a time at which it does not hold.

\begin{figure}
    \centering
    \begin{subfigure}[t]{0.45\textwidth}
    \vspace{0pt}
        \centering \includegraphics[width=0.546\textwidth]{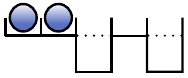}
        \caption{Case of $X_k(t) = \un_{\{\xi_k(t) > \xi_k(0)\}}$: the presence of a particle between the two traps would imply that another particle has been trapped.}\label{fig:nd_1}
    \end{subfigure}
        \hspace{0.08\textwidth} 
    \begin{subfigure}[t]{0.45\textwidth} 
    \vspace{0pt}
        \centering \includegraphics[width=0.754\textwidth]{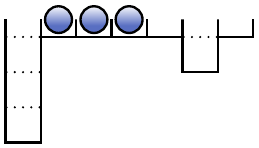}
        \caption{Case of $X_k'(t) = \un_{\{\xi_k(t) =1\}}$. The minimal number of jumps to have $\xi_6 = \xi_7 = 1$ is 10. }\label{fig:nd_2}
    \end{subfigure}
    \caption{Counter-examples of negative dependence}\label{fig:ex_nd}
\end{figure}

\medskip

The first case is quite straightforward, starting from a configuration with an empty zone surrounded by traps. For example , choose $K=5$, and $\xi=(1,1,-1,0,-1)$, as illustrated in Figure \ref{fig:nd_1}. Then, for a particle to have arrived at site $k=4$ in between the traps, the other particle must surely have fallen into one of the two traps, so that for any $t> 0$, 
\begin{align}
\PP_\xi(X_4(t) = 1,  X_3(t) +X_5(t) \geq 1) &= \PP_\xi(X_4(t) = 1) \nonumber \\
&> \PP_\xi(X_4(t) = 1)\PP_\xi(X_3(t)+X_5(t) \geq 1).
\end{align}
and negative dependence is not preserved.

\medskip

For the second case, we choose $X_i'(t) = \un_{\{\xi_i(t) = 1\}}$, and show there is no pairwise negative correlation with the following example.  Set $K=7$, and consider the  initial configuration 
\eq{eq:xiK6}{\xi:=(-3, 1,1,1,0,-1,0),}
represented in Figure \ref{fig:nd_2}, in which there is a big trap at site $1$ (this breaks the periodicity), three particles at sites $2,3,4$, a trap of depth $1$ at site $6$ and two empty sites at sites $5$ and $7$. We show that for $t$ small enough,
\eqs{\PP_{\xi}(\xi_6(t)=\xi_7(t)=1) >\PP_{\xi}(\xi_6(t)=1)\PP_{\xi}(\xi_7(t)=1),}
so that negative correlation is not preserved by the dynamics.

In $\xi$, at least 5 jumps are required to have a live particle at site $6$, by trapping at site $6$ the particle initially at site $4$ (it jumps twice to the right), and then making the particle initially at site $3$ jump thrice to the right.
Similarly, it takes at least $6$ jumps to have a live particle at site $7$.
Lastly, it takes at least 10 jumps to have live particles both at sites $6$ and $7$. 
Define the functions 
\eqs{f(\xi) = \un_{\{\xi_6 = 1\}},\qquad  g(\xi) = \un_{\{\xi_7 = 1\}} \qquad \mbox{ and }\qquad   h(\xi) =  f(\xi) g(\xi).}

Since all our jumps occur at rate $1$,  if $f(\xi)=0$, $\mathscr{L}_7^{\SWT} f(\xi)$ gives the number of possible jumps in $\xi$  leading to a configuration $\xi'$ satisfying $\xi'_6= 1$. More generally, if given $\xi$ 
\begin{itemize}
\item there exists no sequence of at most $k$ jumps leading to a configuration $\xi'$ satisfying $f(\xi') = 1$,
\item there exists a sequence of $k+1$ jumps leading to a configuration $\xi'$ satisfying $f(\xi') = 1$,
\end{itemize}
then for all $l \le k$, $(\mathscr{L}_7^{\SWT})^l f(\xi) = 0$ and $(\mathscr{L}_7^{\SWT})^{k+1} f(\xi)$ is the number of sequences of $k+1$ jumps leading to $f(\xi') = 1$. The same goes for $g$ and $h$.

We thus obtain
\begin{align}
    \forall n < 5, (\mathscr{L}_7^{\SWT})^n f (\xi) &= 0 \nonumber\\
    \forall n < 6, (\mathscr{L}_7^{\SWT})^n g (\xi) &= 0 \nonumber\\
    \forall n < 10, (\mathscr{L}_7^{\SWT})^n h (\xi) &= 0.
\end{align}

Define the semi-group $P_t = e^{t \mathscr{L}_7^{\SWT}} = \sum_{n\ge 0} \frac{t^n}{n!}(\mathscr{L}_7^{\SWT})^n $ associated with $\mathscr{L}_7^{\SWT}$, and note that the operator norm $\|\mathscr{L}_K^{\SWT}\|$  of $\mathscr{L}_K^{\SWT}$ is less than $14$.
Then we have
\begin{align}
    P_t f(\xi) &= \frac{t^{5}}{5!} + \sum_{n\ge 6} \frac{t^n}{n!}(\mathscr{L}_7^{\SWT})^n f(\xi) \nonumber \\
    P_t g(\xi) &= \frac{t^6}{6!} + \sum_{n\ge 7} \frac{t^n}{n!}(\mathscr{L}_7^{\SWT})^n g(\xi) \nonumber \\
    P_t h(\xi) &= \frac{t^{10}}{10!} + \sum_{n\ge 11} \frac{t^n}{n!}(\mathscr{L}_7^{\SWT})^n h(\xi)
\end{align}
In particular, for $t\ll 1/14$ small enough, we have $P_t h(\xi)=\mathcal{O}(t^{10}) > P_t f(\xi) P_t g(\xi)=\mathcal{O}(t^{11}) $, which proves that negative dependence is not preserved by the SWT.

\medskip

Notably, since both the SSEP and the SSEP with reservoirs preserve negative dependence, this means that the only moments where the SWT breaks the preservation of negative dependence is the moment when a trap disappears. Indeed, aside from these moments, in between traps, one observes exactly a SSEP in contact with reservoirs on each ergodic segment delimited by the traps.

\bibliographystyle{plain}
\bibliography{biblio.bib}

\end{document}